\documentclass[11pt]{amsart}
\usepackage[top=1.5in, bottom=1.5in, left=1.5in, right=1.5in]{geometry}
\expandafter\let\csname ver@amsthm.sty\endcsname\relax

\usepackage{amsmath}
\usepackage{amssymb}
\usepackage{hyperref}
\usepackage{amsthm}
\usepackage{mathtools}
\usepackage[capitalize,noabbrev]{cleveref}
\usepackage{xcolor}
\usepackage{enumitem}

\usepackage{tikz}
\usetikzlibrary{decorations.pathmorphing}

\usepackage{ytableau}

\allowdisplaybreaks

\numberwithin{equation}{section}

\newtheorem{thm}{Theorem}[section]
\newtheorem{lemma}[thm]{Lemma}

\newtheorem{Definition}[thm]{Definition}

\newtheorem{Example}[thm]{Example}
\newenvironment{example}
  {\begin{Example}\rm}{\end{Example}}

\newtheorem{Remark}[thm]{Remark}
\newenvironment{remark}
  {\begin{Remark}\rm}{\end{Remark}}
  
\crefname{thm}{Theorem}{Theorems}
\crefname{lemma}{Lemma}{Lemmas}
\crefname{cor}{Corollary}{Corollaries}
\crefname{prop}{Proposition}{Propositions}

\crefname{definition}{Definition}{Definitions}
\crefname{example}{Example}{Examples}
\crefname{remark}{Remark}{Remarks}

\newcommand{\dfn}[1]{\textcolor{blue}{\emph{#1}}}

\newcommand{\mch}[2]{
\left.\mathchoice
  {\left(\kern-0.48em\binom{#1}{#2}\kern-0.48em\right)}
  {\big(\kern-0.30em\binom{\smash{#1}}{\smash{#2}}\kern-0.30em\big)}
  {\left(\kern-0.30em\binom{\smash{#1}}{\smash{#2}}\kern-0.30em\right)}
  {\left(\kern-0.30em\binom{\smash{#1}}{\smash{#2}}\kern-0.30em\right)}
\right.}

\title[Combinatorial reciprocity for non-intersecting paths]{Combinatorial reciprocity \\ for non-intersecting paths}
\author{Sam Hopkins}
\email{samuelfhopkins@gmail.com}
\address{Department of Mathematics, Howard University, Washington, DC, USA}
\author{Gjergji Zaimi}
\email{gjergjiz@gmail.com}
\address{Palo Alto, California, USA}



\begin{document}

\begin{abstract}
We prove a combinatorial reciprocity theorem for the enumeration of non-intersecting paths in a linearly growing sequence of acyclic planar networks. We explain two applications of this theorem: reciprocity for fans of bounded Dyck paths, and reciprocity for Schur function evaluations with repeated values.
\end{abstract}

\maketitle

\section{Introduction}

Suppose that $\mathcal{A}$ is a family of combinatorial objects and $f(n)$ counts the number of $\mathcal{A}$-objects of size $n$ for all $n \geq 0$. A \dfn{combinatorial reciprocity} theorem asserts that $f(-n) = \pm g(n)$ for all $n \geq 0$, where $g(n)$ counts the number of $\mathcal{B}$-objects of size $n$ for some other, related family of combinatorial objects $\mathcal{B}$. Combinatorial reciprocity is a hidden duality which certain enumeration problems possess. The  prototypical example of combinatorial reciprocity is
\begin{equation} \label{eqn:ur_reciprocity}
    \binom{-n}{k} = (-1)^k \mch{n}{k},
\end{equation} 
where on the left the binomial coefficient $\binom{n}{k}$ of course counts the number of $k$-element subsets of $[n] \coloneqq \{1,2,\ldots,n\}$, and on the right the multichoose number~$\mch{n}{k}$ counts the number of $k$-element multisets on $[n]$. For a survey of the many other known examples of combinatorial reciprocity, consult Stanley's classic paper~\cite{stanley1974combinatorial} or the book by Beck and Sanyal~\cite{beck2018combinatorial}.

Here we prove a combinatorial reciprocity theorem for the enumeration of non-intersecting tuples of paths in a linearly growing sequence of acyclic planar networks. 

However, before we discuss non-intersecting paths, let us take a moment to highlight one subtlety regarding our definition of combinatorial reciprocity: it refers to~$f(-n)$, when the counting function $f$ is a priori defined only at nonnegative integers. Thus, to establish a combinatorial reciprocity result, we also have to supply a sensible way to evaluate our counting function at negative integers. For~\eqref{eqn:ur_reciprocity}, we do this by realizing that
\[ \binom{n}{k} = \frac{n(n-1)\ldots(n-k+1)}{k!}\]
is in fact a polynomial in $n$. This polynomial can then be evaluated at negative integers to give a meaning to $\binom{-n}{k}$. 

For many combinatorial reciprocity theorems, the relevant counting functions are polynomials. For instance, this is the case with Ehrhart-Macdonald reciprocity for Ehrhart polynomials of lattice polytopes~\cite{ehrhart1967demonstration, macdonald1971polynomials}. In some cases, however, the counting function~$f(n)$ may not be given by a polynomial in $n$, but we can still make sense of $f(-n)$ and prove an interesting reciprocity result. For example, combinatorial reciprocity has been studied for certain perfect matching enumeration problems where the counting function has terms that are exponential in $n$~\cite{propp2001reciprocity, speyer2001reciprocity}. A broad class of functions which are a priori only defined only for nonnegative integers but can be extended to negative integers, including both polynomials like in~\cite{ehrhart1967demonstration, macdonald1971polynomials} and exponential functions like in~\cite{propp2001reciprocity, speyer2001reciprocity}, are those satisfying a linear recurrence.\footnote{An even more general class of functions than those satisfying a linear recurrence are the $P$-recursive functions, and Stanley has studied reciprocity for $P$-recursive functions in~\cite{stanley1980differentiably}.}

As mentioned, in this paper we prove a combinatorial reciprocity theorem for the enumeration of non-intersecting paths in a linearly growing network. Specifically, we fix an acyclic planar network $G$ and let $G^n$ be the network obtained by gluing $n$ copies of $G$ together. We show the function of $n$ which counts the number of non-intersecting paths in $G^n$ connecting a given pattern of sources to sinks satisfies a linear recurrence. Furthermore, assuming one more technical condition on the network $G$, we show that  the evaluation of this function at $-n$ yields the number of non-intersecting paths in~$G^n$ connecting the complementary pattern of sources to sinks. In fact, our result allows the network to have edge weights and hence considers the weighted enumeration of paths. See \cref{thm:main} for the exact statement.

Non-intersecting paths occupy a central role in modern enumerative combinatorics, mainly for two reasons. The first reason is that the Lindstr\"{o}m--Gessel--Viennot (LGV) lemma~\cite{lindstrom1973vector, gessel1989determinants} gives a determinantal formula for their enumeration. Unsurprisingly, the LGV lemma is a key input to our proof of the reciprocity theorem. Indeed, beyond the LGV lemma, the proof is linear algebra.

The second reason non-intersecting paths are ubiquitous in modern combinatorics is that they are remarkably flexible objects, which can represent many other kinds of discrete structures such as plane partitions, tableaux, et cetera. Thus, our combinatorial reciprocity result specializes to reciprocity statements for these various other structures. We discuss in detail two specific applications of our main result: reciprocity for fans of bounded Dyck paths, and reciprocity for Schur function evaluations with repeated values. The reciprocity result for fans of bounded Dyck paths was recently obtained by Cigler and Krattenthaler~\cite{cigler2020bounded} (see also the follow-up work by Jang et al.~\cite{jang2022negative}). Meanwhile, the reciprocity result for Schur functions is also a known result in the context of symmetric function theory: see~\cite[Exercise 24]{stanley2022ec2supp}. Our present work gives a new graphical approach to these results. 

The rest of the paper is structured as follows: in \cref{sec:main} we prove the main reciprocity result for non-intersecting paths, and in \cref{sec:apps} we discuss the aforementioned applications to fans of bounded Dyck paths and Schur function evaluations with repeated values.

\subsection*{Acknowledgments} This research originated on MathOverflow~\cite{MO}, where the first author asked a question about reciprocity for fans of bounded Dyck paths which the second author answered. We thank Johann Cigler, whose prior MathOverflow questions about bounded Dyck paths were thus crucial to the genesis of our collaboration (see also the introduction of~\cite{cigler2020bounded} for more on the history of Cigler's investigations). We also thank Richard Stanley for pointing out~\cite[Exercise 24]{stanley2022ec2supp} to us. Finally, we thank the anonymous referee for their attention to our manuscript.

\section{Main result} \label{sec:main}

Throughout we work over the complex numbers $\mathbb{C}$ for convenience and concreteness. Thus all matrices have entries in $\mathbb{C}$, all functions are valued in $\mathbb{C}$, and so on. It is reasonable to ask how much of our work extends from $\mathbb{C}$ to other commutative rings, but we do not pursue such questions here.

A \dfn{planar network} $G$ consists of the following data:
\begin{itemize}
    \item a finite directed graph $G=(V,E)$ drawn in a disc in a planar manner (i.e., with edges intersecting only at vertices);
    \item a choice of $m$ distinguished \dfn{source vertices} $s_1,\ldots,s_m \in V$ of indegree zero and $m$ distinguished \dfn{sink vertices} $t_1,\ldots,t_m \in V$ of outdegree zero that are all on the boundary of the disc, and which occur in the order $s_1$, $s_2$, ..., $s_m$, $t_m$, $t_{m-1}$, ..., $t_1$ as we traverse the boundary clockwise;
    \item an \dfn{edge weight} $w(e) \in \mathbb{C}$ for each edge $e \in E$.
\end{itemize}
The planar network $G$ is \dfn{acyclic} if it contains no directed cycles. From now on in this section we fix an acyclic planar network $G$. To avoid degeneracies we assume that the sources and sinks of $G$ are disjoint.

A \dfn{path} in $G$ is a sequence $\pi = v_0, e_1, v_1, \ldots, v_{\ell-1}, e_{\ell}, v_{\ell}$ of vertices $v_0,\ldots,v_{\ell} \in V$ and edges $e_1,\ldots,e_{\ell} \in E$ such that $e_i = (v_{i-1},v_i)$ for all~$i=1,\ldots,\ell$. (That $G$ is acyclic guarantees that all vertices and all edges in $\pi$ are distinct.) We say that this path is \dfn{from~$v_0$ to~$v_{\ell}$} and write $\pi\colon v_0 \to v_{\ell}$. The \dfn{weight} of such a path is
\[ w(\pi) \coloneqq \prod_{i=1}^{\ell} w(e_i),\]
where by convention this product is $1$ if $\ell=0$.

Now let $\Pi=(\pi_1,\ldots,\pi_k)$ be a tuple of paths in $G$. The \dfn{weight} of such a $\Pi$ is 
\[ w(\Pi) = \prod_{i=1}^{k} w(\pi),\]
where again this product is $1$ if $k=0$. We write $\Pi\colon(u_1,\ldots,u_k) \to (v_1,\ldots,v_k)$ to mean $\pi_i\colon u_i \to v_i$ for all $i=1,\ldots,k$. We say that $\Pi$ is \dfn{non-intersecting} if paths $\pi_i$ and $\pi_j$ do not have any vertices in common for all $i \neq j$. 

The \dfn{path matrix} $\mathsf{P}_G$ of $G$ is the $m \times m$ matrix whose $i,j$ entry, for $1 \leq i,j \leq m$, is
\[ (\mathsf{P}_G)_{i,j} \coloneqq \sum_{\pi}w(\pi),\]
a sum over paths $\pi\colon s_i \to t_j$ in $G$. (That $G$ is acyclic guarantees that this sum is finite.) If there are no paths $\pi\colon s_i \to t_j$ then $(\mathsf{P}_G)_{i,j}=0$.

Let $\mathsf{M}$ be an $m \times m$ matrix. We use $\det(\mathsf{M})$ to denote the determinant of $\mathsf{M}$ (which is $1$ if $m=0$). For two $k$-element subsets $I,J\subseteq [m]$, we use $\mathsf{M}[I,J]$ to denote the~$k \times k$ matrix obtained from $\mathsf{M}$ by restricting to rows~$I$ and columns~$J$. The famous \dfn{Lindstr\"{o}m--Gessel--Viennot (LGV) lemma}~\cite{lindstrom1973vector, gessel1989determinants} is the following:

\begin{lemma} [Lindstr\"{o}m--Gessel--Viennot lemma] \label{lem:LGV}
For any $I=\{i_1 < i_2 <\ldots < i_k\}$, $J=\{j_1 < j_2 < \cdots < j_k\} \subseteq [m]$, we have
\[ \det (\mathsf{P}_G[I,J]) = \sum_{\Pi} w(\Pi),\]
a sum over non-intersecting tuples of paths $\Pi\colon (s_{i_1},\ldots,s_{i_k}) \to (t_{j_1},\ldots,t_{j_k})$ in $G$. 
\end{lemma}

See also \cite[Theorem~2.7.1]{stanley2012ec1} and \cite[Theorem~3.1.14]{ardila2015algebraic} for discussion of the LGV lemma. Sometimes the LGV lemma refers to a more general result where the network need not be planar and with the sinks and sources arranged in the manner we have required. But for other networks the tuples of paths we sum over will come with signs in addition to weights. Our conditions guarantee that all signs are positive, because non-intersecting paths must connect source $s_{i_\ell}$ to sink $t_{j_\ell}$. 

\begin{figure}
    \begin{tikzpicture}
    \draw[very thick, black, fill=blue!25] (0,0) ellipse (0.5 and 1.5);
    \node[inner sep=1,fill=red,draw=red,circle,label=left:{$s_1$}] at (-0.37,-1) {};
    \node[inner sep=1,fill=red,draw=red,circle,label=left:{$s_2$}] at (-0.47,-0.5) {};
    \node[label=left:{$\vdots$}] at (-0.5,0.2) {};
    \node[inner sep=1,fill=red,draw=red,circle,label=left:{$s_m$}] at (-0.45,0.75) {};
    \node[inner sep=1,fill=red,draw=red,circle,label=left:{\tiny $t_1$}] at (0.37,-1) {};
    \node[inner sep=1,fill=red,draw=red,circle,label=left:{\tiny $t_2$}] at (0.47,-0.5) {};
    \node[label=right:{$\vdots$}] at (0.5,0.2) {};
    \node[inner sep=1,fill=red,draw=red,circle,label=left:{\tiny $t_m$}] at (0.45,0.75) {};
    \node[label={$G$}] at (0,-0.3) {};
    \draw[very thick, black, fill=blue!25] (1.6,0) ellipse (0.5 and 1.5);
    \node[inner sep=1,fill=red,draw=red,circle,label={[xshift=-0.125 cm]right:{\tiny $s_1$}}] at (1.6+-0.37,-1) {};
    \node[inner sep=1,fill=red,draw=red,circle,label={[xshift=-0.125 cm]right:{\tiny $s_2$}}] at (1.6+-0.47,-0.5) {};
    \node[inner sep=1,fill=red,draw=red,circle,label={[xshift=-0.125 cm]right:{\tiny $s_m$}}] at (1.6+-0.45,0.75) {};
    \node[label={$G$}] at (1.6+0,-0.3) {};
    \draw[red, decorate, decoration=snake] (0.37,-1) to (1.6+-0.37,-1);
    \draw[red, decorate, decoration=snake] (0.47,-0.5) to (1.6+-0.47,-0.5);
    \draw[red, decorate, decoration=snake] (0.45,0.75) to (1.6+-0.45,0.75);
    \node[inner sep=1,fill=red,draw=red,circle,label={[xshift=0.125 cm]left:{\tiny $t_1$}}] at (1.6+0.37,-1) {};
    \node[inner sep=1,fill=red,draw=red,circle,label={[xshift=0.125 cm]left:{\tiny $t_2$}}] at (1.6+0.47,-0.5) {};
    \node[inner sep=1,fill=red,draw=red,circle,label={[xshift=0.125 cm]left:{\tiny $t_m$}}] at (1.6+0.45,0.75) {};
    \draw[red, decorate, decoration=snake] (1.6+0.37,-1) to (1.6+1,-1);
    \draw[red, decorate, decoration=snake] (1.6+0.47,-0.5) to (1.6+1,-0.5);
    \draw[red, decorate, decoration=snake] (1.6+0.45,0.75) to (1.6+1,0.75);
    \node at (3,0) {\huge $\cdots$};
    \draw[very thick, black, fill=blue!25] (4.3,0) ellipse (0.5 and 1.5);
    \node[inner sep=1,fill=red,draw=red,circle,label=right:{\tiny $s_1$}] at (4.3+-0.37,-1) {};
    \node[inner sep=1,fill=red,draw=red,circle,label=right:{\tiny $s_2$}] at (4.3+-0.47,-0.5) {};
    \node[inner sep=1,fill=red,draw=red,circle,label=right:{\tiny $s_m$}] at (4.3+-0.45,0.75) {};
    \draw[red, decorate, decoration=snake] (4.3-1,-1) to (4.3+-0.37,-1);
    \draw[red, decorate, decoration=snake] (4.3-1,-0.5) to (4.3+-0.47,-0.5);
    \draw[red, decorate, decoration=snake] (4.3-1,0.75) to (4.3+-0.45,0.75);
    \node[label={$G$}] at (4.3+0,-0.3) {};
    \node[inner sep=1,fill=red,draw=red,circle,label=right:{$t_1$}] at (4.3+0.37,-1) {};
    \node[inner sep=1,fill=red,draw=red,circle,label=right:{$t_2$}] at (4.3+0.47,-0.5) {};
    \node[label=right:{$\vdots$}] at (4.3+0.5,0.2) {};
    \node[inner sep=1,fill=red,draw=red,circle,label=right:{$t_m$}] at (4.3+0.45,0.75) {};
    \end{tikzpicture}
    \caption{The graph $G^n$ consisting of $n$ copies of $G$ glued together. The wavy red lines denote identification of vertices.}
    \label{fig:gn}
\end{figure}
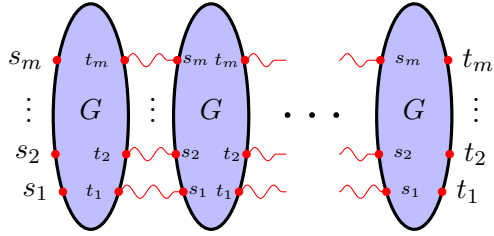

Next, we define a linearly growing sequence $G^{0}, G^{1}, G^{2}, \ldots$ of acyclic planar networks, where $G^n$ consists of $n$ copies of $G$ glued together. More precisely, we let~$G^0$ be the degenerate network which has $m$ vertices $s_1 = t_1, s_2 = t_2, \ldots, s_m = t_m$ and no edges. For $n \geq 1$, we let network $G^n$ be obtained from network $G^{n-1}$ by appending a copy of~$G$ to $G^{n-1}$, while identifying the sinks $t_1,\ldots,t_m$ in $G^{n-1}$ with the sources $s_1,\ldots,s_m$ in this new copy of $G$. Except in the case $n=1$, these vertices become internal -- i.e., neither sources nor sinks -- after this identification. See~\cref{fig:gn} for a diagrammatic representation of $G^n$. It is clear that $G^n$ remains planar (with sinks and sources arranged appropriately) and acyclic. Also, notice that $G^1=G$.

The final thing we need to do before we can state our main result is review what it means for a function to satisfy a linear recurrence, and explain how such a function naturally extends to negative inputs. Let $f\colon \mathbb{N} \to \mathbb{C}$ be a function defined on the nonnegative integers $\mathbb{N}$. We say that~$f$ satisfies a \dfn{(homogeneous) linear recurrence (with constant coefficients)} if there exists a~$d \geq 0$ and $\alpha_1, \alpha_2, \ldots, \alpha_d \in \mathbb{C}$ with~$\alpha_d \neq 0$ such that
\[ f(n+d) + \alpha_1 f(n+d-1) + \alpha_2 f(n+d-2) + \cdots + \alpha_d f(n) = 0\]
for all $n \geq 0$. Such an $f$ is determined by this recurrence together with the initial values $f(0),f(1),\ldots,f(d-1)$. Thus, for such an $f$, we can define the value of $f$ at negative integers by ``running the recurrence backwards'' from the same $d$ initial values. That is, for all $n\geq 1$, we set 
\[ f(-n) \coloneqq \frac{-1}{\alpha_d} \big(f(-n+d) + \alpha_1 f(-n+d-1) + \cdots + \alpha_{d-1} f(-n+1)\big). \]

\begin{remark} \label{rem:neg_vals}
It is well known (see, e.g.,~\cite[Theorem 4.1.1(i)]{stanley2012ec1}) that $f\colon \mathbb{N} \to \mathbb{C}$ satisfying a linear recurrence is equivalent to the generating function
\[ F(x) \coloneqq \sum_{n \geq 0} f(n) \; x^n\] being a rational function $F(x) = P(x)/Q(x)$ where the degree of the polynomial~$P(x) \in \mathbb{C}[x]$ is strictly less than that of $Q(x) \in \mathbb{C}[x]$. We may equivalently (see~\cite[Proposition 4.2.3]{stanley2012ec1}) define the value of such an $f$ at negative integers by setting
\[ \sum_{n \geq 1} f(-n) \; x^{n} \coloneqq -F(x^{-1}),\]
an equality of rational functions. Note that this generating function perspective makes it clear that if $f$ satisfies multiple linear recurrences, it does not matter which we use to extend $f$ to negative integers: they must all give the same result.

It is also well known (see, e.g.,~\cite[Theorem 4.1.1(iii)]{stanley2012ec1}) that such an $f$ is a combination of polynomials and exponential functions in the sense that there exist polynomials $P_1(x),\ldots,P_k(x) \in \mathbb{C}[x]$ and complex numbers $\gamma_1,\ldots,\gamma_k \in \mathbb{C} \setminus \{0\}$ for which
\[f(n) = P_1(n) \; \gamma_1^n + \cdots + P_k(n) \; \gamma_k^n\]
for all $n \geq 0$. We may again equivalently define $f$ at negative integers by naively plugging $-n$ into this formula.
\end{remark}

Let $I\subseteq [m]$ be any subset. We use $\sigma(I) \coloneqq \sum_{i \in I} i$ to denote the sum of the elements in~$I$, and $I^c \coloneqq [m]\setminus I$ to denote the complementary subset. Our main result is the following:

\begin{thm} \label{thm:main}
Suppose that $\det(\mathsf{P}_G) \neq 0$. For subsets $I=\{i_1 < i_2 <\ldots < i_k\}$, $J=\{j_1 < j_2 < \cdots < j_k\} \subseteq [m]$ and $n \geq 0$, define
\[ f_G(I,J;n) \coloneqq  \sum_{\Pi} w(\Pi), \]
a sum over non-intersecting tuples of paths $\Pi \colon (s_{i_1},\ldots,s_{i_k}) \to (t_{j_1},\ldots,t_{j_k})$ in $G^n$. Then, for fixed $I$ and $J$, $f_G(I,J;n)$ satisfies a linear recurrence as a function of~$n$. Moreover, suppose that $\det(\mathsf{P}_G) = 1$. Then for all $n \geq 1$,
\[ f_G(I,J;-n) = (-1)^{\sigma(I) + \sigma(J)} \; f_G(J^c,I^c;n).\]
\end{thm}

\begin{remark} \label{rem:determinant}
In practice, the conditions $\det(\mathsf{P}_G)\neq 0$ and $\det(\mathsf{P}_G)=1$ in \cref{thm:main} can be verified using the LGV lemma. For example, to show $\det(\mathsf{P}_G)=1$ we check that there is a unique non-intersecting tuple of paths connecting all the sources to all the sinks, and that every edge in this tuple has weight one. 

We also note that for any $G$ with $\det(\mathsf{P}_G) \neq 0$, the proof of \cref{thm:main} will give
\[f_G(I,J;-n) = (-1)^{\sigma(I)+\sigma(J)} \, \det(\mathsf{P}_G)^{-n} \, f_G(J^c,I^c;n).\]
When $\det(\mathsf{P}_G)\neq1$, the right-hand side of this equality does not have a direct combinatorial meaning, but it is a ratio of terms with combinatorial meaning.
\end{remark}

\begin{remark}
Both the acyclicity condition, and the planarity condition, can be relaxed in the LGV lemma in various ways (see, e.g.,~\cite{fomin2001loop, talaska2008formula,  talaska2012determinants}). We could similarly relax these conditions for \cref{thm:main}. However, we chose to state the theorem in the cleanest way possible. As we will see in \cref{sec:apps}, the theorem as stated suffices for our motivating applications.
\end{remark}

\begin{remark} \label{rem:speyer}
\Cref{thm:main} is very similar in form to a reciprocity result for the enumeration of perfect matchings in a linearly growing sequence of graphs that was proved by Speyer in the unpublished manuscript~\cite{speyer2001reciprocity}. While it is known that non-intersecting path enumeration problems can often be translated to perfect matching enumeration problems (see, e.g., \cite[\S3.3]{kuperberg2002kasteleyn}), we do not see any way to directly deduce \cref{thm:main} from Speyer's theorem. It would certainly be interesting to find a common generalization of these two reciprocity theorems.
\end{remark}

Though we just stated in \cref{rem:speyer} that we are not exactly sure how \cref{thm:main} is related to Speyer's theorem in~\cite{speyer2001reciprocity}, the starting points for the two proofs are the same. We will need the following lemma explaining how the sequence of particular entries of powers of a fixed matrix compares to the same sequence for its inverse. (See also~\cite[\S5]{jang2022negative} for a similar application of matrix inverses to reciprocity.)

\begin{lemma} \label{lem:matrix_rec}
Let $\mathsf{M}$ be an invertible $m \times m$ matrix and fix $i,j$ with $1 \leq i,j \leq m$. Define $f\colon \mathbb{N} \to \mathbb{C}$ by $f(n)\coloneqq (\mathsf{M}^n)_{i,j}$ for all $n\geq 0$. Then $f$ satisfies a linear recurrence, and $f(-n) = (\mathsf{M}^{-n})_{i,j}$ for all $n \geq 1$.
\end{lemma}

\begin{proof}
As mentioned, this lemma appears as~\cite[Claim 1]{speyer2001reciprocity}. There, Speyer proves this lemma via the Cayley--Hamilton theorem (i.e., the fact that $\mathsf{M}$ satisfies its own characteristic polynomial). Let us explain a slightly different approach via generating functions. Here we closely follow~\cite[proof of Theorem 4.7.2]{stanley2012ec1}.

Let $f(n)$ be as in the statement of the lemma and let $F(x) \coloneqq \sum_{n \geq 0} f(n) \, x^n$. Observe that $F(x) = ((\mathsf{I} - x \mathsf{M})^{-1})_{i,j}$ where $\mathsf{I}$ is the $m\times m$ identity matrix. Thus, by the standard cofactor formula for the entries of the inverse of a matrix, we have
\[ F(x) = \frac{(-1)^{i+j} \, \det(\mathsf{I} - x \mathsf{M}: j,i)}{\det(\mathsf{I} - x \mathsf{M})}.\]
Here $(\mathsf{N}: i,j)$ is the matrix we get by deleting the $i$th row and $j$th column of $\mathsf{N}$.

If the characteristic polynomial of $\mathsf{M}$ is $\det(x\mathsf{I}-\mathsf{M}) = x^m + \alpha_1 \, x^{m-1} + \cdots + \alpha_{k} \, x^{m-k} $ then $\det(\mathsf{I}-x\mathsf{M}) = 1+\alpha_1 \, x + \cdots + \alpha_k \, x^k$. In particular, if $\mathsf{M}$ is invertible, so that $0$ is not a root of the characteristic polynomial of $\mathsf{M}$, then the degree of $\det(\mathsf{I}-x\mathsf{M})$ is~$m$. Meanwhile, the degree of $\det(\mathsf{I} - x \mathsf{M}: j,i)$ is less than or equal to $m-1$. Hence, $F(x)$ is a rational function where the numerator has degree strictly less than the denominator. So by the discussion in \cref{rem:neg_vals}, $f(n)$ satisfies a linear recurrence.

Now consider the rational function $-F(x^{-1})$. We have 
\begin{align*}
-F(x^{-1}) &= -( (\mathsf{I}-x^{-1}\mathsf{M})^{-1})_{i,j} \\
&= (x\mathsf{M}^{-1}(\mathsf{I}-x\mathsf{M}^{-1})^{-1})_{i,j} \\
&= \sum_{n\geq 1}(\mathsf{M}^{-n})_{i,j} \, x^n.
\end{align*}
By the discussion in \cref{rem:neg_vals}, we conclude $f(-n) = (\mathsf{M}^{-n})_{i,j}$ for all $n \geq 0$.
\end{proof}

The other main result from linear algebra we will need in our proof of \cref{thm:main} is the relationship between the compound and adjugate matrices. Let $\mathsf{M}$ be an $m \times m$ matrix, and let $0 \leq k \leq m$. The \dfn{$k$th compound matrix} of $\mathsf{M}$, denoted $\mathrm{com}_k(\mathsf{M})$, and \dfn{$k$th adjugate matrix} of $\mathsf{M}$, denoted $\mathrm{adj}_k(\mathsf{M})$, are two $\binom{m}{k}\times\binom{m}{k}$ matrices whose entries are minors of $\mathsf{M}$. More precisely, the rows and columns of these matrices are indexed by $k$-element subsets of $[m]$ ordered lexicographically, and for $k$-element subsets $I,J\subseteq [m]$, the corresponding entries of these matrices are
\begin{align*}
\mathrm{com}_k(\mathsf{M})_{I,J} &\coloneqq \det(\mathsf{M}[I,J]) \\
\mathrm{adj}_k(\mathsf{M})_{I,J} &\coloneqq (-1)^{\sigma(I)+\sigma(J)} \, \det(\mathsf{M}[J^c,I^c])
\end{align*}
The following lemma explains the significance of these two matrices.

\begin{lemma} \label{lem:com_adj}
Let $\mathsf{M}$ be an $m \times m$ matrix. Then, for any $0 \leq k \leq m$,
\[ \mathrm{com}_k(\mathsf{M}) \, \mathrm{adj}_k(\mathsf{M}) = \mathrm{adj}_k(\mathsf{M}) \, \mathrm{com}_k(\mathsf{M}) = \det(\mathsf{M}) \cdot \mathsf{I} \]
where $\mathsf{I}$ is the $\binom{m}{k}\times\binom{m}{k}$ identity matrix.
\end{lemma}

\begin{proof}
This is essentially an expression of the (generalized) Laplace cofactor expansion of the determinant. See for instance \cite[Chapter V, \S38]{aitken1954determinants}.
\end{proof}

We can now prove our main result.

\begin{proof}[Proof of \cref{thm:main}]
Because $G^n$ consists of $n$ copies of $G$ glued together sequentially, any $k$-tuple of non-intersecting paths in $G^n$ connecting sources to sinks decomposes as a concatenation of $n$ such tuples of paths in the various copies of $G$. By the LGV lemma (\cref{lem:LGV}), the entries of $\mathrm{com}_k(\mathsf{P}_{G^n})$ and $\mathrm{com}_k(\mathsf{P}_{G})$ record the generating functions of such tuples. Thus, $\mathrm{com}_k(\mathsf{P}_{G^n})=\mathrm{com}_k(\mathsf{P}_{G})^n$. By the same logic, and being careful with the sign, we also have $\mathrm{adj}_k(\mathsf{P}_{G^n})=\mathrm{adj}_k(\mathsf{P}_{G})^n$.

For $k$-element subsets $I,J\subseteq[m]$ we have, again by the LGV lemma, that
\[f(I,J;n)=\mathrm{com}_k(\mathsf{P}_{G^n})_{I,J}=(\mathrm{com}_k(\mathsf{P}_{G})^n)_{I,J}.\]
From \cref{lem:com_adj} we see that the supposition $\det(\mathsf{P}_{G})\neq 0$ implies that $\mathrm{com}_k(\mathsf{P}_{G})$ is invertible. Hence, \cref{lem:matrix_rec} says that $f_G(I,J;n)$ satisfies a linear recurrence.

Now suppose $\det(\mathsf{P}_{G})=1$. Then \cref{lem:com_adj} says that $\mathrm{com}_k(\mathsf{P}_{G})^{-1} = \mathrm{adj}_k(\mathsf{P}_{G})$. Hence, by \cref{lem:matrix_rec},
\[f_G(I,J;-n) = (\mathrm{adj}_k(\mathsf{P}_{G})^n)_{I,J} = \mathrm{adj}_k(\mathsf{P}_{G^n})_{I,J} = (-1)^{\sigma(I)+\sigma(J)} \, f_G(J^c,I^c;n)\]
where for the last equality we appeal once more to the LGV lemma.
\end{proof}

\begin{remark}
A natural question is to what extent the proof of \cref{thm:main} can be made combinatorial. The standard proof of the LGV lemma is via a sign-reversing involution, and so appeals to the LGV lemma can in principle always be made combinatorial. \Cref{lem:matrix_rec} more-or-less follows from the Cayley--Hamilton theorem, for which various combinatorial proofs exist (see, e.g., \cite{straubing1983combinatorial}). On the other hand, we do not know of any combinatorial proof of \cref{lem:com_adj}. Indeed, we are not aware of any bijection realizing the identity in \cref{lem:com_adj} even in the special case of path matrices, where the entries of the compound and adjugate matrices have a clear combinatorial meaning. It would be interesting to construct such a bijection.
\end{remark}

\section{Applications} \label{sec:apps}

In this section we review some specific networks where \cref{thm:main} yields a combinatorial reciprocity result of particular significance.

\subsection{Reciprocity for fans of bounded Dyck paths} \label{sec:dyck}

A \dfn{Dyck path} of semilength~$n$ is a sequence $D = (x_0,y_0), (x_1,y_1),\ldots, (x_{2n},y_{2n})$ of lattice points $(x_i,y_i) \in \mathbb{Z}^2$ such that:
\begin{itemize}
\item the path is from $(x_0,y_0)=(0,0)$ to $(x_{2n},y_{2n})=(2n,0)$;
\item for $i=1,\ldots,2n$, each step $(x_i,y_i)-(x_{i-1},y_{i-1})$ of the path is either an \dfn{up step} of the form $(1,1)$ or a \dfn{down step} of the form $(1,-1)$;
\item the path never goes below the $x$-axis, i.e., $y_i \geq 0$ for $i=0,\ldots,2n$.
\end{itemize}
Necessarily a Dyck path of semilength $n$ has $n$ up steps and $n$ down steps. The number of Dyck paths of semilength $n$ is the famous \dfn{Catalan number} $C_n 
= \frac{1}{n+1}\binom{2n}{n}$. 

Let $D = (x_0,y_0),\ldots,(x_{2n},y_{2n})$ and $D' = (x'_0,y'_0),\ldots,(x'_{2n},y'_{2n})$ be two Dyck paths. We write $D \leq D'$ to mean that $D$ stays weakly below~$D'$, i.e., that $y_i \leq y'_i$ for $i=0,\ldots,2n$. An \dfn{$m$-fan of Dyck paths} is an $m$-tuple $(D_1,\ldots,D_m)$ of Dyck paths for which $D_1 \leq D_2 \leq \cdots \leq D_m$. 

\begin{figure}
\begin{tikzpicture}
\node at (0,0) {\scalebox{0.75}{\begin{tikzpicture}
\node at (0,0) {\begin{tikzpicture}
\draw[red,fill=red] (0,0) circle (0.1);
\draw[red,fill=red] (1,1) circle (0.1);
\draw[red,fill=red] (2,0) circle (0.1);
\draw[red,fill=red] (3,1) circle (0.1);
\draw[red,fill=red] (4,0) circle (0.1);
\draw[red,fill=red] (5,1) circle (0.1);
\draw[red,fill=red] (6,2) circle (0.1);
\draw[red,fill=red] (7,1) circle (0.1);
\draw[red,fill=red] (8,0) circle (0.1);
\draw[very thick, red] (0,0) -- (1,1) -- (2,0) -- (3,1) -- (4,0) -- (5,1) -- (6,2) -- (7,1) -- (8,0);
\node at (4,4) {};
\end{tikzpicture}};
\node at (0,0.25) {\begin{tikzpicture}
\draw[purple,fill=purple] (0,0) circle (0.1);
\draw[purple,fill=purple] (1,1) circle (0.1);
\draw[purple,fill=purple] (2,0) circle (0.1);
\draw[purple,fill=purple] (3,1) circle (0.1);
\draw[purple,fill=purple] (4,2) circle (0.1);
\draw[purple,fill=purple] (5,1) circle (0.1);
\draw[purple,fill=purple] (6,2) circle (0.1);
\draw[purple,fill=purple] (7,1) circle (0.1);
\draw[purple,fill=purple] (8,0) circle (0.1);
\draw[very thick, purple] (0,0) -- (1,1) -- (2,0) -- (3,1) -- (4,2) -- (5,1) -- (6,2) -- (7,1) -- (8,0);
\node at (4,4) {};
\end{tikzpicture}};
\node at (0,0.5) {\begin{tikzpicture}
\draw[blue,fill=blue] (0,0) circle (0.1);
\draw[blue,fill=blue] (1,1) circle (0.1);
\draw[blue,fill=blue] (2,0) circle (0.1);
\draw[blue,fill=blue] (3,1) circle (0.1);
\draw[blue,fill=blue] (4,2) circle (0.1);
\draw[blue,fill=blue] (5,1) circle (0.1);
\draw[blue,fill=blue] (6,2) circle (0.1);
\draw[blue,fill=blue] (7,1) circle (0.1);
\draw[blue,fill=blue] (8,0) circle (0.1);
\draw[very thick, blue] (0,0) -- (1,1) -- (2,0) -- (3,1) -- (4,2) -- (5,1) -- (6,2) -- (7,1) -- (8,0);
\node at (4,4) {};
\end{tikzpicture}};
\node at (0,0.75) {\begin{tikzpicture}
\draw[green,fill=green] (0,0) circle (0.1);
\draw[green,fill=green] (1,1) circle (0.1);
\draw[green,fill=green] (2,2) circle (0.1);
\draw[green,fill=green] (3,1) circle (0.1);
\draw[green,fill=green] (4,2) circle (0.1);
\draw[green,fill=green] (5,3) circle (0.1);
\draw[green,fill=green] (6,2) circle (0.1);
\draw[green,fill=green] (7,1) circle (0.1);
\draw[green,fill=green] (8,0) circle (0.1);
\draw[very thick, green] (0,0) -- (1,1) -- (2,2) -- (3,1) -- (4,2) -- (5,3) -- (6,2) -- (7,1) -- (8,0);
\node at (4,4) {};
\end{tikzpicture}};
\node at (0,1) {\begin{tikzpicture}
\draw[yellow,fill=yellow] (0,0) circle (0.1);
\draw[yellow,fill=yellow] (1,1) circle (0.1);
\draw[yellow,fill=yellow] (2,2) circle (0.1);
\draw[yellow,fill=yellow] (3,1) circle (0.1);
\draw[yellow,fill=yellow] (4,2) circle (0.1);
\draw[yellow,fill=yellow] (5,3) circle (0.1);
\draw[yellow,fill=yellow] (6,2) circle (0.1);
\draw[yellow,fill=yellow] (7,1) circle (0.1);
\draw[yellow,fill=yellow] (8,0) circle (0.1);
\draw[very thick, yellow] (0,0) -- (1,1) -- (2,2) -- (3,1) -- (4,2) -- (5,3) -- (6,2) -- (7,1) -- (8,0);
\node at (4,4) {};
\end{tikzpicture}};
\end{tikzpicture}}};
\end{tikzpicture} \qquad \vrule \hspace{-1cm} \begin{tikzpicture}
\node at (0,0) {\ytableausetup{boxsize=2em}
\begin{ytableau}
5 & 3 & 0 \\
5 & 1 \\
3
\end{ytableau}};
\node at (-0.35,0.4) {\rotatebox{45}{\scalebox{0.54}{\begin{tikzpicture}
\node at (0,0) {\begin{tikzpicture}
\draw[red,fill=red] (0,0) circle (0.1);
\draw[red,fill=red] (1,1) circle (0.1);
\draw[red,fill=red] (2,0) circle (0.1);
\draw[red,fill=red] (3,1) circle (0.1);
\draw[red,fill=red] (4,0) circle (0.1);
\draw[red,fill=red] (5,1) circle (0.1);
\draw[red,fill=red] (6,2) circle (0.1);
\draw[red,fill=red] (7,1) circle (0.1);
\draw[red,fill=red] (8,0) circle (0.1);
\draw[very thick, red] (0,0) -- (1,1) -- (2,0) -- (3,1) -- (4,0) -- (5,1) -- (6,2) -- (7,1) -- (8,0);
\node at (4,4) {};
\end{tikzpicture}};
\node at (0,0.15) {\begin{tikzpicture}
\draw[purple,fill=purple] (0,0) circle (0.1);
\draw[purple,fill=purple] (1,1) circle (0.1);
\draw[purple,fill=purple] (2,0) circle (0.1);
\draw[purple,fill=purple] (3,1) circle (0.1);
\draw[purple,fill=purple] (4,2) circle (0.1);
\draw[purple,fill=purple] (5,1) circle (0.1);
\draw[purple,fill=purple] (6,2) circle (0.1);
\draw[purple,fill=purple] (7,1) circle (0.1);
\draw[purple,fill=purple] (8,0) circle (0.1);
\draw[very thick, purple] (0,0) -- (1,1) -- (2,0) -- (3,1) -- (4,2) -- (5,1) -- (6,2) -- (7,1) -- (8,0);
\node at (4,4) {};
\end{tikzpicture}};
\node at (0,0.3) {\begin{tikzpicture}
\draw[blue,fill=blue] (0,0) circle (0.1);
\draw[blue,fill=blue] (1,1) circle (0.1);
\draw[blue,fill=blue] (2,0) circle (0.1);
\draw[blue,fill=blue] (3,1) circle (0.1);
\draw[blue,fill=blue] (4,2) circle (0.1);
\draw[blue,fill=blue] (5,1) circle (0.1);
\draw[blue,fill=blue] (6,2) circle (0.1);
\draw[blue,fill=blue] (7,1) circle (0.1);
\draw[blue,fill=blue] (8,0) circle (0.1);
\draw[very thick, blue] (0,0) -- (1,1) -- (2,0) -- (3,1) -- (4,2) -- (5,1) -- (6,2) -- (7,1) -- (8,0);
\node at (4,4) {};
\end{tikzpicture}};
\node at (0,0.45) {\begin{tikzpicture}
\draw[green,fill=green] (0,0) circle (0.1);
\draw[green,fill=green] (1,1) circle (0.1);
\draw[green,fill=green] (2,2) circle (0.1);
\draw[green,fill=green] (3,1) circle (0.1);
\draw[green,fill=green] (4,2) circle (0.1);
\draw[green,fill=green] (5,3) circle (0.1);
\draw[green,fill=green] (6,2) circle (0.1);
\draw[green,fill=green] (7,1) circle (0.1);
\draw[green,fill=green] (8,0) circle (0.1);
\draw[very thick, green] (0,0) -- (1,1) -- (2,2) -- (3,1) -- (4,2) -- (5,3) -- (6,2) -- (7,1) -- (8,0);
\node at (4,4) {};
\end{tikzpicture}};
\node at (0,0.6) {\begin{tikzpicture}
\draw[yellow,fill=yellow] (0,0) circle (0.1);
\draw[yellow,fill=yellow] (1,1) circle (0.1);
\draw[yellow,fill=yellow] (2,2) circle (0.1);
\draw[yellow,fill=yellow] (3,1) circle (0.1);
\draw[yellow,fill=yellow] (4,2) circle (0.1);
\draw[yellow,fill=yellow] (5,3) circle (0.1);
\draw[yellow,fill=yellow] (6,2) circle (0.1);
\draw[yellow,fill=yellow] (7,1) circle (0.1);
\draw[yellow,fill=yellow] (8,0) circle (0.1);
\draw[very thick, yellow] (0,0) -- (1,1) -- (2,2) -- (3,1) -- (4,2) -- (5,3) -- (6,2) -- (7,1) -- (8,0);
\node at (4,4) {};
\end{tikzpicture}};
\end{tikzpicture}}}};
\end{tikzpicture}
\caption{The bijection from $5$-fans of Dyck paths of semilength~$4$ (left) to plane partitions of shape $\delta_4$ with entries in~$\{0,\ldots,5\}$ (right). Under this bijection, each entry of the plane partition records the number of Dyck paths below and to the right of its box.}
\label{fig:pparts}
\end{figure}
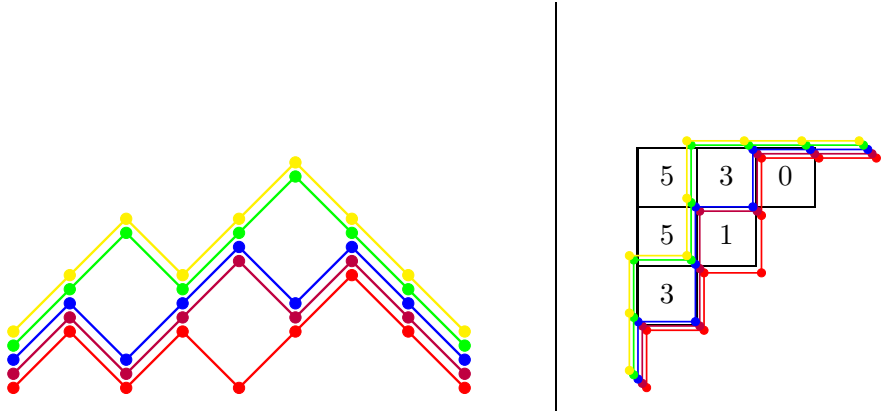

\begin{remark} \label{rem:plane_partitions}
Fans of Dyck paths have an interpretation in terms of plane partitions, which we now briefly explain. Let $\lambda$ be a partition.\footnote{For basics on partitions, Young diagrams, et cetera, see \cref{sec:schur} below.} A \dfn{plane partition of shape $\lambda$} is a filling of the Young diagram of~$\lambda$ with nonnegative integers that is weakly decreasing along rows and down columns. There is a natural bijection between $m$-fans of Dyck paths of semilength~$n$ and plane partitions of \dfn{staircase shape} $\delta_{n}\coloneqq (n-1,n-2,\ldots, 1)$ with entries in~$\{0,1,\ldots,m\}$. This bijection is depicted in \cref{fig:pparts}. 

The number of plane partitions of shape $\delta_n$ with entries in $\{0,\ldots,m\}$ is
\[\prod_{1\leq i <  j \leq n} \frac{2m+i+j-1}{i+j-1} \]
This product formula was first obtained by Proctor~\cite[Corollary 4.1]{proctor1988odd}; see also~\cite[Exercise 7.101(a)]{stanley1999ec2}.\footnote{And furthermore note that these plane partitions constitute Class~6 in Stanley's taxonomy of symmetry classes of plane partitions in a box~\cite{stanley1986symmetries}.} Via the lattice path interpretation of these plane partitions, and using the LGV lemma, this formula amounts to the evaluation of a certain determinant: see \cite[\S3.1.6, Example~4]{ardila2015algebraic} and~\cite[Theorem~26]{krattenthaler1999advanced}.
\end{remark}

We are interested in fans of \emph{bounded} Dyck paths. For $r \geq 1$, we say that the Dyck path $D=(x_0,y_0),\ldots,(x_{2n},y_{2n})$ is \dfn{$r$-bounded} if it never goes above the horizontal line $y=r$, i.e., if $y_i \leq r$ for all $i=0,\ldots,2n$.

For $m,k,n \geq 0$, define $d(m,k;n)$ to be the number of $m$-fans of $(2k+1)$-bounded Dyck paths of semilength $n$. Cigler and Krattenthaler~\cite{cigler2020bounded} obtained the following reciprocity result for the numbers $d(m,k;n)$.

\begin{thm}[{Cigler--Krattenthaler~\cite{cigler2020bounded}}] \label{thm:dyck}
For fixed $m$ and $k$, $d(m,k;n)$ satisfies a linear recurrence as a function of $n$, and for $n\geq 1$ we have
\[ d(m,k;-n) = d(k,m;n+1).\]
\end{thm}

\begin{figure}
\begin{tikzpicture}
\draw[dashed] (-5.5*0.75,1.825*0.75) -- (5.5*0.75,1.825*0.75);
\node at (0,0) {\scalebox{0.75}{\begin{tikzpicture}
\node at (0,0) {\begin{tikzpicture}
\draw[red,fill=red] (0,0) circle (0.1);
\draw[red,fill=red] (1,1) circle (0.1);
\draw[red,fill=red] (2,0) circle (0.1);
\draw[red,fill=red] (3,1) circle (0.1);
\draw[red,fill=red] (4,0) circle (0.1);
\draw[red,fill=red] (5,1) circle (0.1);
\draw[red,fill=red] (6,2) circle (0.1);
\draw[red,fill=red] (7,1) circle (0.1);
\draw[red,fill=red] (8,0) circle (0.1);
\draw[red,fill=red] (9,1) circle (0.1);
\draw[red,fill=red] (10,0) circle (0.1);
\draw[very thick, red] (0,0) -- (1,1) -- (2,0) -- (3,1) -- (4,0) -- (5,1) -- (6,2) -- (7,1) -- (8,0) -- (9,1) -- (10,0);
\node at (5,3) {};
\end{tikzpicture}};
\node at (0,0.25) {\begin{tikzpicture}
\draw[purple,fill=purple] (0,0) circle (0.1);
\draw[purple,fill=purple] (1,1) circle (0.1);
\draw[purple,fill=purple] (2,0) circle (0.1);
\draw[purple,fill=purple] (3,1) circle (0.1);
\draw[purple,fill=purple] (4,2) circle (0.1);
\draw[purple,fill=purple] (5,3) circle (0.1);
\draw[purple,fill=purple] (6,2) circle (0.1);
\draw[purple,fill=purple] (7,1) circle (0.1);
\draw[purple,fill=purple] (8,2) circle (0.1);
\draw[purple,fill=purple] (9,1) circle (0.1);
\draw[purple,fill=purple] (10,0) circle (0.1);
\draw[very thick, purple] (0,0) -- (1,1) -- (2,0) -- (3,1) -- (4,2) -- (5,3) -- (6,2) -- (7,1) -- (8,2) -- (9,1) -- (10,0);
\node at (5,3) {};
\end{tikzpicture}};
\node at (0,0.5) {\begin{tikzpicture}
\draw[blue,fill=blue] (0,0) circle (0.1);
\draw[blue,fill=blue] (1,1) circle (0.1);
\draw[blue,fill=blue] (2,0) circle (0.1);
\draw[blue,fill=blue] (3,1) circle (0.1);
\draw[blue,fill=blue] (4,2) circle (0.1);
\draw[blue,fill=blue] (5,3) circle (0.1);
\draw[blue,fill=blue] (6,2) circle (0.1);
\draw[blue,fill=blue] (7,1) circle (0.1);
\draw[blue,fill=blue] (8,2) circle (0.1);
\draw[blue,fill=blue] (9,1) circle (0.1);
\draw[blue,fill=blue] (10,0) circle (0.1);
\draw[very thick, blue] (0,0) -- (1,1) -- (2,0) -- (3,1) -- (4,2) -- (5,3) -- (6,2) -- (7,1) -- (8,2) -- (9,1) -- (10,0);
\node at (5,3) {};
\end{tikzpicture}};
\node at (0,0.75) {\begin{tikzpicture}
\draw[green,fill=green] (0,0) circle (0.1);
\draw[green,fill=green] (1,1) circle (0.1);
\draw[green,fill=green] (2,2) circle (0.1);
\draw[green,fill=green] (3,1) circle (0.1);
\draw[green,fill=green] (4,2) circle (0.1);
\draw[green,fill=green] (5,3) circle (0.1);
\draw[green,fill=green] (6,2) circle (0.1);
\draw[green,fill=green] (7,3) circle (0.1);
\draw[green,fill=green] (8,2) circle (0.1);
\draw[green,fill=green] (9,1) circle (0.1);
\draw[green,fill=green] (10,0) circle (0.1);
\draw[very thick, green] (0,0) -- (1,1) -- (2,2) -- (3,1) -- (4,2) -- (5,3) -- (6,2) -- (7,3) -- (8,2) -- (9,1) -- (10,0);
\node at (5,3) {};
\end{tikzpicture}};
\end{tikzpicture}}};
\end{tikzpicture} \quad \vrule 
\hspace{-1cm} \begin{tikzpicture}
\node at (0,0) {\ytableausetup{boxsize=2em}
\begin{ytableau}
\none & \none & 3 & 1 \\
\none & 1 & 0 \\
4 & 1 \\
3
\end{ytableau}};
\node at (-0.14,0.20) {\rotatebox{45}{\scalebox{0.56}{\begin{tikzpicture}
\node at (0,0) {\begin{tikzpicture}
\draw[red,fill=red] (0,0) circle (0.1);
\draw[red,fill=red] (1,1) circle (0.1);
\draw[red,fill=red] (2,0) circle (0.1);
\draw[red,fill=red] (3,1) circle (0.1);
\draw[red,fill=red] (4,0) circle (0.1);
\draw[red,fill=red] (5,1) circle (0.1);
\draw[red,fill=red] (6,2) circle (0.1);
\draw[red,fill=red] (7,1) circle (0.1);
\draw[red,fill=red] (8,0) circle (0.1);
\draw[red,fill=red] (9,1) circle (0.1);
\draw[red,fill=red] (10,0) circle (0.1);
\draw[very thick, red] (0,0) -- (1,1) -- (2,0) -- (3,1) -- (4,0) -- (5,1) -- (6,2) -- (7,1) -- (8,0) -- (9,1) -- (10,0);
\node at (5,3) {};
\end{tikzpicture}};
\node at (0,0.15) {\begin{tikzpicture}
\draw[purple,fill=purple] (0,0) circle (0.1);
\draw[purple,fill=purple] (1,1) circle (0.1);
\draw[purple,fill=purple] (2,0) circle (0.1);
\draw[purple,fill=purple] (3,1) circle (0.1);
\draw[purple,fill=purple] (4,2) circle (0.1);
\draw[purple,fill=purple] (5,3) circle (0.1);
\draw[purple,fill=purple] (6,2) circle (0.1);
\draw[purple,fill=purple] (7,1) circle (0.1);
\draw[purple,fill=purple] (8,2) circle (0.1);
\draw[purple,fill=purple] (9,1) circle (0.1);
\draw[purple,fill=purple] (10,0) circle (0.1);
\draw[very thick, purple] (0,0) -- (1,1) -- (2,0) -- (3,1) -- (4,2) -- (5,3) -- (6,2) -- (7,1) -- (8,2) -- (9,1) -- (10,0);
\node at (5,3) {};
\end{tikzpicture}};
\node at (0,0.3) {\begin{tikzpicture}
\draw[blue,fill=blue] (0,0) circle (0.1);
\draw[blue,fill=blue] (1,1) circle (0.1);
\draw[blue,fill=blue] (2,0) circle (0.1);
\draw[blue,fill=blue] (3,1) circle (0.1);
\draw[blue,fill=blue] (4,2) circle (0.1);
\draw[blue,fill=blue] (5,3) circle (0.1);
\draw[blue,fill=blue] (6,2) circle (0.1);
\draw[blue,fill=blue] (7,1) circle (0.1);
\draw[blue,fill=blue] (8,2) circle (0.1);
\draw[blue,fill=blue] (9,1) circle (0.1);
\draw[blue,fill=blue] (10,0) circle (0.1);
\draw[very thick, blue] (0,0) -- (1,1) -- (2,0) -- (3,1) -- (4,2) -- (5,3) -- (6,2) -- (7,1) -- (8,2) -- (9,1) -- (10,0);
\node at (5,3) {};
\end{tikzpicture}};
\node at (0,0.45) {\begin{tikzpicture}
\draw[green,fill=green] (0,0) circle (0.1);
\draw[green,fill=green] (1,1) circle (0.1);
\draw[green,fill=green] (2,2) circle (0.1);
\draw[green,fill=green] (3,1) circle (0.1);
\draw[green,fill=green] (4,2) circle (0.1);
\draw[green,fill=green] (5,3) circle (0.1);
\draw[green,fill=green] (6,2) circle (0.1);
\draw[green,fill=green] (7,3) circle (0.1);
\draw[green,fill=green] (8,2) circle (0.1);
\draw[green,fill=green] (9,1) circle (0.1);
\draw[green,fill=green] (10,0) circle (0.1);
\draw[very thick, green] (0,0) -- (1,1) -- (2,2) -- (3,1) -- (4,2) -- (5,3) -- (6,2) -- (7,3) -- (8,2) -- (9,1) -- (10,0);
\node at (5,3) {};
\end{tikzpicture}};
\end{tikzpicture}}}};
\end{tikzpicture}
\caption{The bijection from $4$-fans of $3$-bounded Dyck paths of semilength $5$ (left) to plane partitions of shape $\delta_5 / \delta_3$ with entries in~$\{0,\ldots,4\}$ (right). This plane partition gives the bounded alternating sequence $3 \leq 4 \geq 1 \leq 1 \geq 0 \leq 3 \geq 1$.}
\label{fig:pparts_bounded}
\end{figure}
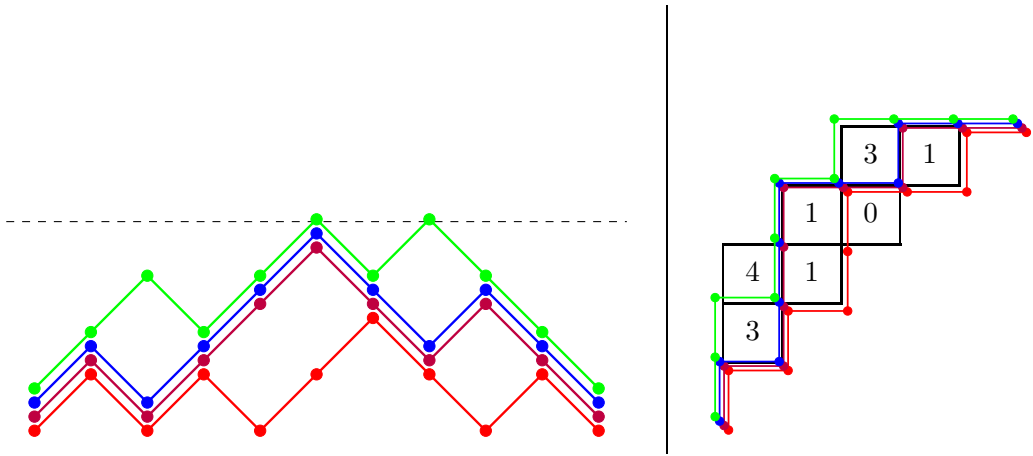

\begin{example}
Consider the case $m=k=1$. Thus, $d(1,1;n)$ counts the number of $3$-bounded Dyck paths of semilength $n$. It is not hard to see directly that this number is $d(1,1;n) = F_{2n-1}$, where the~$F_n$ are the famous \dfn{Fibonacci numbers} defined by~$F_1=F_2=1$ and $F_n = F_{n-1} + F_{n-2}$ for $n > 2$. Because of the well-known explicit formula for the Fibonacci numbers, this means that $d(1,1;n) = \frac{1}{\sqrt{5}} (\varphi^{2n-1} + \varphi^{-2n+1} )$ where $\varphi=\frac{1+\sqrt{5}}{2}$. \Cref{thm:dyck} says we should have $d(1,1;-n) = d(1,1;n+1)$, and indeed this is easy to verify directly from the formula involving $\varphi$.
\end{example}

\begin{remark}
Via the bijection between fans of Dyck paths and plane partitions discussed in \cref{rem:plane_partitions}, $m$-fans of $r$-bounded Dyck paths of semilength $n$ correspond to plane partitions of \dfn{skew staircase shape} $\delta_{n} / \delta_{n-r+1}$ with entries in $\{0,1,\ldots,m\}$. This is depicted in \cref{fig:pparts_bounded}. In the case $r=3$, by reading the entries of these plane partitions from bottom-to-top and left-to-right we see that they are the same as integer sequences $a_1 \leq a_2 \geq a_3 \leq a_4 \geq a_5 \leq \cdots \geq a_{2n-3}$ with $0\leq a_i \leq m$ for all~$i=1,\ldots,2n-3$.  Cigler and Krattenthaler~\cite{cigler2020bounded} refer to such sequences as \dfn{bounded alternating sequences}. They interpret the case $m=1$ of \cref{thm:dyck} as saying that bounded alternating sequences are the ``negative version'' of bounded Dyck paths.
\end{remark}

We now give a short proof of \cref{thm:dyck} using \cref{thm:main}.

\begin{figure}
\begin{tikzpicture}[scale=0.75]
\node[inner sep=1.5,fill=black,draw=black,circle,label=left:{$s_1$}] (1) at (0,0) {};
\node[inner sep=1.5,fill=black,draw=black,circle,label=left:{$s_2$}] (2) at (0,2) {};
\node[inner sep=1.5,fill=black,draw=black,circle,label=left:{$s_3$}] (3) at (0,4) {};
\node[inner sep=1.5,fill=black,draw=black,circle,label=left:{\Large $\vdots$}] (4) at (0,6) {};
\node[inner sep=1.5,fill=black,draw=black,circle,label=left:{$s_{m+k}$}] (5) at (0,8) {};
\node[inner sep=1.5,fill=black,draw=black,circle] (6) at (1,1) {};
\node[inner sep=1.5,fill=black,draw=black,circle] (7) at (1,3) {};
\node[inner sep=1.5,fill=black,draw=black,circle] (8) at (1,5) {};
\node[inner sep=1.5,fill=black,draw=black,circle] (9) at (1,7) {};
\node[inner sep=1.5,fill=black,draw=black,circle] (10) at (1,9) {};
\node[inner sep=1.5,fill=black,draw=black,circle,label=right:{$t_1$}] (11) at (2,0) {};
\node[inner sep=1.5,fill=black,draw=black,circle,label=right:{$t_2$}] (12) at (2,2) {};
\node[inner sep=1.5,fill=black,draw=black,circle,label=right:{$t_3$}] (13) at (2,4) {};
\node[inner sep=1.5,fill=black,draw=black,circle,label=right:{\Large $\vdots$}] (14) at (2,6) {};
\node[inner sep=1.5,fill=black,draw=black,circle,label=right:{$t_{m+k}$}] (15) at (2,8) {};
\draw[->,very thick] (1) -- (6);
\draw[->,very thick] (2) -- (6);
\draw[->,very thick] (2) -- (7);
\draw[->,very thick] (3) -- (7);
\draw[->,very thick] (3) -- (8);
\draw[->,very thick] (4) -- (8);
\draw[->,very thick] (4) -- (9);
\draw[->,very thick] (5) -- (9);
\draw[->,very thick] (5) -- (10);
\draw[->,very thick] (6) -- (11);
\draw[->,very thick] (6) -- (12);
\draw[->,very thick] (7) -- (12);
\draw[->,very thick] (7) -- (13);
\draw[->,very thick] (8) -- (13);
\draw[->,very thick] (8) -- (14);
\draw[->,very thick] (9) -- (14);
\draw[->,very thick] (9) -- (15);
\draw[->,very thick] (10) -- (15);
\end{tikzpicture}
\caption{The acyclic planar network $G$ used in the proof of \cref{thm:dyck}.}
\label{fig:dyck_network}
\end{figure}
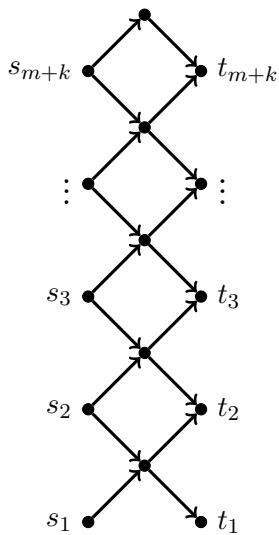

\begin{proof} [Proof of \cref{thm:dyck}]
Let $G$ be the acyclic planar network depicted in \cref{fig:dyck_network}. Notice that $G$ has sources $s_1,\ldots,s_{m+k}$ and sinks $t_1,\ldots,t_{m+k}$, and all edge weights are equal to one. There is a unique non-intersecting tuple of paths in $G$ connecting all of the sources to all of the sinks and hence, as discussed in \cref{rem:determinant}, we have that~$\det(\mathsf{P}_G) = 1$. 

Non-intersecting tuples of paths $\Pi\colon (s_1,\ldots,s_m) \to (t_1,\ldots,t_m)$ in $G^n$ are in bijection with $m$-fans of $k$-bounded Dyck paths of semilength $n$: this is depicted in \cref{fig:dyck_paths}. Hence, $d(m,k;n) = f_{G}([m],[m];n)$. So, by \cref{thm:main}, $d(m,k;n)$ satisfies a linear recurrence as a function of $n$.

Moreover, by \cref{thm:main}, $d(m,k;-n)=f_{G}([m+k]\setminus[m],[m+k]\setminus[m];n)$. And non-intersecting tuples of paths $\Pi\colon (s_{m+1},\ldots,s_{m+k}) \to (t_{m+1},\ldots,t_{m+k})$ in~$G^n$ are in bijection with non-intersecting tuples of paths $\Pi\colon (s_1,\ldots,s_k) \to (t_1,\ldots,t_k)$ in~$G^{n+1}$: this is depicted in \cref{fig:dyck_paths_flip}. Altogether, we conclude as desired that $d(m,k;-n)=f_{G}([k],[k];n+1)=d(k,m;n+1)$.
\end{proof}

\begin{figure}
\begin{tikzpicture}[scale=0.5]
\node[inner sep=1,fill=black,draw=black,circle,label=left:{$s_1$}] (A1) at (0,0) {};
\node[inner sep=1,fill=black,draw=black,circle,label=left:{$s_2$}] (A2) at (0,2) {};
\node[inner sep=1,fill=black,draw=black,circle,label=left:{$s_3$}] (A3) at (0,4) {};
\node[inner sep=1,fill=black,draw=black,circle,label=left:{$s_4$}] (A4) at (0,6) {};
\node[inner sep=1,fill=black,draw=black,circle,label=left:{$s_5$}] (A5) at (0,8) {};
\node[inner sep=1,fill=black,draw=black,circle] (B1) at (1,1) {};
\node[inner sep=1,fill=black,draw=black,circle] (B2) at (1,3) {};
\node[inner sep=1,fill=black,draw=black,circle] (B3) at (1,5) {};
\node[inner sep=1,fill=black,draw=black,circle] (B4) at (1,7) {};
\node[inner sep=1,fill=black,draw=black,circle] (B5) at (1,9) {};
\node[inner sep=1,fill=black,draw=black,circle] (C1) at (2,0) {};
\node[inner sep=1,fill=black,draw=black,circle] (C2) at (2,2) {};
\node[inner sep=1,fill=black,draw=black,circle] (C3) at (2,4) {};
\node[inner sep=1,fill=black,draw=black,circle] (C4) at (2,6) {};
\node[inner sep=1,fill=black,draw=black,circle] (C5) at (2,8) {};
\node[inner sep=1,fill=black,draw=black,circle] (D1) at (3,1) {};
\node[inner sep=1,fill=black,draw=black,circle] (D2) at (3,3) {};
\node[inner sep=1,fill=black,draw=black,circle] (D3) at (3,5) {};
\node[inner sep=1,fill=black,draw=black,circle] (D4) at (3,7) {};
\node[inner sep=1,fill=black,draw=black,circle] (D5) at (3,9) {};
\node[inner sep=1,fill=black,draw=black,circle] (E1) at (4,0) {};
\node[inner sep=1,fill=black,draw=black,circle] (E2) at (4,2) {};
\node[inner sep=1,fill=black,draw=black,circle] (E3) at (4,4) {};
\node[inner sep=1,fill=black,draw=black,circle] (E4) at (4,6) {};
\node[inner sep=1,fill=black,draw=black,circle] (E5) at (4,8) {};
\node[inner sep=1,fill=black,draw=black,circle] (F1) at (5,1) {};
\node[inner sep=1,fill=black,draw=black,circle] (F2) at (5,3) {};
\node[inner sep=1,fill=black,draw=black,circle] (F3) at (5,5) {};
\node[inner sep=1,fill=black,draw=black,circle] (F4) at (5,7) {};
\node[inner sep=1,fill=black,draw=black,circle] (F5) at (5,9) {};
\node[inner sep=1,fill=black,draw=black,circle] (G1) at (6,0) {};
\node[inner sep=1,fill=black,draw=black,circle] (G2) at (6,2) {};
\node[inner sep=1,fill=black,draw=black,circle] (G3) at (6,4) {};
\node[inner sep=1,fill=black,draw=black,circle] (G4) at (6,6) {};
\node[inner sep=1,fill=black,draw=black,circle] (G5) at (6,8) {};
\node[inner sep=1,fill=black,draw=black,circle] (H1) at (7,1) {};
\node[inner sep=1,fill=black,draw=black,circle] (H2) at (7,3) {};
\node[inner sep=1,fill=black,draw=black,circle] (H3) at (7,5) {};
\node[inner sep=1,fill=black,draw=black,circle] (H4) at (7,7) {};
\node[inner sep=1,fill=black,draw=black,circle] (H5) at (7,9) {};
\node[inner sep=1,fill=black,draw=black,circle] (I1) at (8,0) {};
\node[inner sep=1,fill=black,draw=black,circle] (I2) at (8,2) {};
\node[inner sep=1,fill=black,draw=black,circle] (I3) at (8,4) {};
\node[inner sep=1,fill=black,draw=black,circle] (I4) at (8,6) {};
\node[inner sep=1,fill=black,draw=black,circle] (I5) at (8,8) {};
\node[inner sep=1,fill=black,draw=black,circle] (J1) at (9,1) {};
\node[inner sep=1,fill=black,draw=black,circle] (J2) at (9,3) {};
\node[inner sep=1,fill=black,draw=black,circle] (J3) at (9,5) {};
\node[inner sep=1,fill=black,draw=black,circle] (J4) at (9,7) {};
\node[inner sep=1,fill=black,draw=black,circle] (J5) at (9,9) {};
\node[inner sep=1,fill=black,draw=black,circle] (K1) at (10,0) {};
\node[inner sep=1,fill=black,draw=black,circle] (K2) at (10,2) {};
\node[inner sep=1,fill=black,draw=black,circle] (K3) at (10,4) {};
\node[inner sep=1,fill=black,draw=black,circle] (K4) at (10,6) {};
\node[inner sep=1,fill=black,draw=black,circle] (K5) at (10,8) {};
\node[inner sep=1,fill=black,draw=black,circle] (L1) at (11,1) {};
\node[inner sep=1,fill=black,draw=black,circle] (L2) at (11,3) {};
\node[inner sep=1,fill=black,draw=black,circle] (L3) at (11,5) {};
\node[inner sep=1,fill=black,draw=black,circle] (L4) at (11,7) {};
\node[inner sep=1,fill=black,draw=black,circle] (L5) at (11,9) {};
\node[inner sep=1,fill=black,draw=black,circle,label=right:{$t_1$}] (M1) at (12,0) {};
\node[inner sep=1,fill=black,draw=black,circle,label=right:{$t_2$}] (M2) at (12,2) {};
\node[inner sep=1,fill=black,draw=black,circle,label=right:{$t_3$}] (M3) at (12,4) {};
\node[inner sep=1,fill=black,draw=black,circle,label=right:{$t_4$}] (M4) at (12,6) {};
\node[inner sep=1,fill=black,draw=black,circle,label=right:{$t_5$}] (M5) at (12,8) {};
\draw (A1) -- (B1) -- (C1) -- (D1) -- (E1) -- (F1) -- (G1) -- (H1) -- (I1) -- (J1) -- (K1) -- (L1) -- (M1);
\draw (A2) -- (B1) -- (C2) -- (D1) -- (E2) -- (F1) -- (G2) -- (H1) -- (I2) -- (J1) -- (K2) -- (L1) -- (M2);
\draw (A2) -- (B2) -- (C2) -- (D2) -- (E2) -- (F2) -- (G2) -- (H2) -- (I2) -- (J2) -- (K2) -- (L2) -- (M2);
\draw (A3) -- (B2) -- (C3) -- (D2) -- (E3) -- (F2) -- (G3) -- (H2) -- (I3) -- (J2) -- (K3) -- (L2) -- (M3);
\draw (A3) -- (B3) -- (C3) -- (D3) -- (E3) -- (F3) -- (G3) -- (H3) -- (I3) -- (J3) -- (K3) -- (L3) -- (M3);
\draw (A4) -- (B3) -- (C4) -- (D3) -- (E4) -- (F3) -- (G4) -- (H3) -- (I4) -- (J3) -- (K4) -- (L3) -- (M4);
\draw (A4) -- (B4) -- (C4) -- (D4) -- (E4) -- (F4) -- (G4) -- (H4) -- (I4) -- (J4) -- (K4) -- (L4) -- (M4);
\draw (A5) -- (B4) -- (C5) -- (D4) -- (E5) -- (F4) -- (G5) -- (H4) -- (I5) -- (J4) -- (K5) -- (L4) -- (M5);
\draw (A5) -- (B5) -- (C5) -- (D5) -- (E5) -- (F5) -- (G5) -- (H5) -- (I5) -- (J5) -- (K5) -- (L5) -- (M5);
\draw[very thick, red] (A1) -- (B1) -- (C2) -- (D1) -- (E1) -- (F1) -- (G1) -- (H1) -- (I2) -- (J1) -- (K2) -- (L1) -- (M1);
\draw[red,fill=red] (A1) circle (0.2);
\draw[red,fill=red] (B1) circle (0.2);
\draw[red,fill=red] (C2) circle (0.2);
\draw[red,fill=red] (D1) circle (0.2);
\draw[red,fill=red] (E1) circle (0.2);
\draw[red,fill=red] (F1) circle (0.2);
\draw[red,fill=red] (G1) circle (0.2);
\draw[red,fill=red] (H1) circle (0.2);
\draw[red,fill=red] (I2) circle (0.2);
\draw[red,fill=red] (J1) circle (0.2);
\draw[red,fill=red] (K2) circle (0.2);
\draw[red,fill=red] (L1) circle (0.2);
\draw[red,fill=red] (M1) circle (0.2);
\draw[very thick, blue] (A2) -- (B2) -- (C3) -- (D3) -- (E3) -- (F2) -- (G2) -- (H2) -- (I3) -- (J2) -- (K3) -- (L2) -- (M2);
\draw[blue,fill=blue] (A2) circle (0.2);
\draw[blue,fill=blue] (B2) circle (0.2);
\draw[blue,fill=blue] (C3) circle (0.2);
\draw[blue,fill=blue] (D3) circle (0.2);
\draw[blue,fill=blue] (E3) circle (0.2);
\draw[blue,fill=blue] (F2) circle (0.2);
\draw[blue,fill=blue] (G2) circle (0.2);
\draw[blue,fill=blue] (H2) circle (0.2);
\draw[blue,fill=blue] (I3) circle (0.2);
\draw[blue,fill=blue] (J2) circle (0.2);
\draw[blue,fill=blue] (K3) circle (0.2);
\draw[blue,fill=blue] (L2) circle (0.2);
\draw[blue,fill=blue] (M2) circle (0.2);
\draw[very thick, green] (A3) -- (B3) -- (C4) -- (D4) -- (E5) -- (F5) -- (G5) -- (H4) -- (I4) -- (J3) -- (K4) -- (L3) -- (M3);
\draw[green,fill=green] (A3) circle (0.2);
\draw[green,fill=green] (B3) circle (0.2);
\draw[green,fill=green] (C4) circle (0.2);
\draw[green,fill=green] (D4) circle (0.2);
\draw[green,fill=green] (E5) circle (0.2);
\draw[green,fill=green] (F5) circle (0.2);
\draw[green,fill=green] (G5) circle (0.2);
\draw[green,fill=green] (H4) circle (0.2);
\draw[green,fill=green] (I4) circle (0.2);
\draw[green,fill=green] (J3) circle (0.2);
\draw[green,fill=green] (K4) circle (0.2);
\draw[green,fill=green] (L3) circle (0.2);
\draw[green,fill=green] (M3) circle (0.2);
\end{tikzpicture} \vrule \, \begin{tikzpicture}
\node at (0,-2) {};
\draw[dashed] (-6.25*0.5,2.5*0.5+0.5) -- (6.25*0.5,2.5*0.5+0.5);
\node at (0,0) {\begin{tikzpicture}[scale=0.5]
\draw[very thick, red] (0,0) -- (1,1) -- (2,2) -- (3,1) -- (4,0) -- (5,1) -- (6,0) -- (7,1) -- (8,2) -- (9,1) -- (10,2) -- (11,1) -- (12,0);
\draw[red,fill=red] (0,0) circle (0.1);
\draw[red,fill=red] (1,1) circle (0.1);
\draw[red,fill=red] (2,2) circle (0.1);
\draw[red,fill=red] (3,1) circle (0.1);
\draw[red,fill=red] (4,0) circle (0.1);
\draw[red,fill=red] (5,1) circle (0.1);
\draw[red,fill=red] (6,0) circle (0.1);
\draw[red,fill=red] (7,1) circle (0.1);
\draw[red,fill=red] (8,2) circle (0.1);
\draw[red,fill=red] (9,1) circle (0.1);
\draw[red,fill=red] (10,2) circle (0.1);
\draw[red,fill=red] (11,1) circle (0.1);
\draw[red,fill=red] (12,0) circle (0.1);
\node at (6,5) {};
\end{tikzpicture}};
\node at (0,0.25) {\begin{tikzpicture}[scale=0.5]
\draw[very thick, blue] (0,0) -- (1,1) -- (2,2) -- (3,3) -- (4,2) -- (5,1) -- (6,0) -- (7,1) -- (8,2) -- (9,1) -- (10,2) -- (11,1) -- (12,0);
\draw[blue,fill=blue] (0,0) circle (0.1);
\draw[blue,fill=blue] (1,1) circle (0.1);
\draw[blue,fill=blue] (2,2) circle (0.1);
\draw[blue,fill=blue] (3,3) circle (0.1);
\draw[blue,fill=blue] (4,2) circle (0.1);
\draw[blue,fill=blue] (5,1) circle (0.1);
\draw[blue,fill=blue] (6,0) circle (0.1);
\draw[blue,fill=blue] (7,1) circle (0.1);
\draw[blue,fill=blue] (8,2) circle (0.1);
\draw[blue,fill=blue] (9,1) circle (0.1);
\draw[blue,fill=blue] (10,2) circle (0.1);
\draw[blue,fill=blue] (11,1) circle (0.1);
\draw[blue,fill=blue] (12,0) circle (0.1);
\node at (6,5) {};
\end{tikzpicture}};
\node at (0,0.5) {\begin{tikzpicture}[scale=0.5]
\draw[very thick, green] (0,0) -- (1,1) -- (2,2) -- (3,3) -- (4,4) -- (5,5) -- (6,4) -- (7,3) -- (8,2) -- (9,1) -- (10,2) -- (11,1) -- (12,0);
\draw[green,fill=green] (0,0) circle (0.1);
\draw[green,fill=green] (1,1) circle (0.1);
\draw[green,fill=green] (2,2) circle (0.1);
\draw[green,fill=green] (3,3) circle (0.1);
\draw[green,fill=green] (4,4) circle (0.1);
\draw[green,fill=green] (5,5) circle (0.1);
\draw[green,fill=green] (6,4) circle (0.1);
\draw[green,fill=green] (7,3) circle (0.1);
\draw[green,fill=green] (8,2) circle (0.1);
\draw[green,fill=green] (9,1) circle (0.1);
\draw[green,fill=green] (10,2) circle (0.1);
\draw[green,fill=green] (11,1) circle (0.1);
\draw[green,fill=green] (12,0) circle (0.1);
\node at (6,5) {};
\end{tikzpicture}};
\end{tikzpicture}
\caption{In the proof of \cref{thm:main}: the bijection from non-intersecting tuples of paths $\Pi\colon (s_1,s_2,s_3) \to (t_1,t_2,t_3)$ in $G^6$ (left) to $3$-fans of $5$-bounded Dyck paths of semilength $6$ (right). The bijection simply lowers all paths down to the same starting height.}
\label{fig:dyck_paths}
\end{figure}
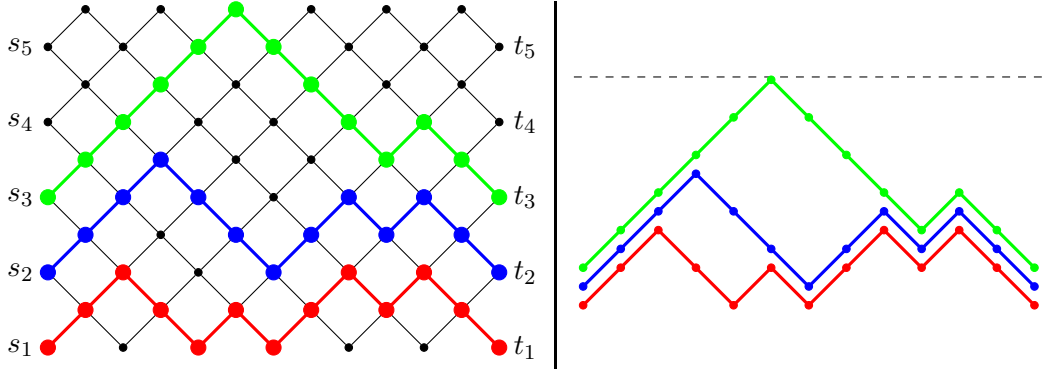

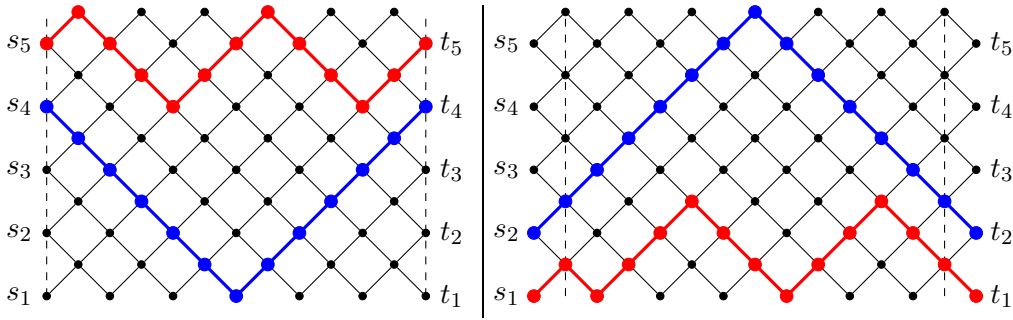
\begin{figure}
\begin{tikzpicture}[scale=0.42]
\node[inner sep=1,fill=black,draw=black,circle,label=left:{$s_1$}] (A1) at (0,0) {};
\node[inner sep=1,fill=black,draw=black,circle,label=left:{$s_2$}] (A2) at (0,2) {};
\node[inner sep=1,fill=black,draw=black,circle,label=left:{$s_3$}] (A3) at (0,4) {};
\node[inner sep=1,fill=black,draw=black,circle,label=left:{$s_4$}] (A4) at (0,6) {};
\node[inner sep=1,fill=black,draw=black,circle,label=left:{$s_5$}] (A5) at (0,8) {};
\node[inner sep=1,fill=black,draw=black,circle] (B1) at (1,1) {};
\node[inner sep=1,fill=black,draw=black,circle] (B2) at (1,3) {};
\node[inner sep=1,fill=black,draw=black,circle] (B3) at (1,5) {};
\node[inner sep=1,fill=black,draw=black,circle] (B4) at (1,7) {};
\node[inner sep=1,fill=black,draw=black,circle] (B5) at (1,9) {};
\node[inner sep=1,fill=black,draw=black,circle] (C1) at (2,0) {};
\node[inner sep=1,fill=black,draw=black,circle] (C2) at (2,2) {};
\node[inner sep=1,fill=black,draw=black,circle] (C3) at (2,4) {};
\node[inner sep=1,fill=black,draw=black,circle] (C4) at (2,6) {};
\node[inner sep=1,fill=black,draw=black,circle] (C5) at (2,8) {};
\node[inner sep=1,fill=black,draw=black,circle] (D1) at (3,1) {};
\node[inner sep=1,fill=black,draw=black,circle] (D2) at (3,3) {};
\node[inner sep=1,fill=black,draw=black,circle] (D3) at (3,5) {};
\node[inner sep=1,fill=black,draw=black,circle] (D4) at (3,7) {};
\node[inner sep=1,fill=black,draw=black,circle] (D5) at (3,9) {};
\node[inner sep=1,fill=black,draw=black,circle] (E1) at (4,0) {};
\node[inner sep=1,fill=black,draw=black,circle] (E2) at (4,2) {};
\node[inner sep=1,fill=black,draw=black,circle] (E3) at (4,4) {};
\node[inner sep=1,fill=black,draw=black,circle] (E4) at (4,6) {};
\node[inner sep=1,fill=black,draw=black,circle] (E5) at (4,8) {};
\node[inner sep=1,fill=black,draw=black,circle] (F1) at (5,1) {};
\node[inner sep=1,fill=black,draw=black,circle] (F2) at (5,3) {};
\node[inner sep=1,fill=black,draw=black,circle] (F3) at (5,5) {};
\node[inner sep=1,fill=black,draw=black,circle] (F4) at (5,7) {};
\node[inner sep=1,fill=black,draw=black,circle] (F5) at (5,9) {};
\node[inner sep=1,fill=black,draw=black,circle] (G1) at (6,0) {};
\node[inner sep=1,fill=black,draw=black,circle] (G2) at (6,2) {};
\node[inner sep=1,fill=black,draw=black,circle] (G3) at (6,4) {};
\node[inner sep=1,fill=black,draw=black,circle] (G4) at (6,6) {};
\node[inner sep=1,fill=black,draw=black,circle] (G5) at (6,8) {};
\node[inner sep=1,fill=black,draw=black,circle] (H1) at (7,1) {};
\node[inner sep=1,fill=black,draw=black,circle] (H2) at (7,3) {};
\node[inner sep=1,fill=black,draw=black,circle] (H3) at (7,5) {};
\node[inner sep=1,fill=black,draw=black,circle] (H4) at (7,7) {};
\node[inner sep=1,fill=black,draw=black,circle] (H5) at (7,9) {};
\node[inner sep=1,fill=black,draw=black,circle] (I1) at (8,0) {};
\node[inner sep=1,fill=black,draw=black,circle] (I2) at (8,2) {};
\node[inner sep=1,fill=black,draw=black,circle] (I3) at (8,4) {};
\node[inner sep=1,fill=black,draw=black,circle] (I4) at (8,6) {};
\node[inner sep=1,fill=black,draw=black,circle] (I5) at (8,8) {};
\node[inner sep=1,fill=black,draw=black,circle] (J1) at (9,1) {};
\node[inner sep=1,fill=black,draw=black,circle] (J2) at (9,3) {};
\node[inner sep=1,fill=black,draw=black,circle] (J3) at (9,5) {};
\node[inner sep=1,fill=black,draw=black,circle] (J4) at (9,7) {};
\node[inner sep=1,fill=black,draw=black,circle] (J5) at (9,9) {};
\node[inner sep=1,fill=black,draw=black,circle] (K1) at (10,0) {};
\node[inner sep=1,fill=black,draw=black,circle] (K2) at (10,2) {};
\node[inner sep=1,fill=black,draw=black,circle] (K3) at (10,4) {};
\node[inner sep=1,fill=black,draw=black,circle] (K4) at (10,6) {};
\node[inner sep=1,fill=black,draw=black,circle] (K5) at (10,8) {};
\node[inner sep=1,fill=black,draw=black,circle] (L1) at (11,1) {};
\node[inner sep=1,fill=black,draw=black,circle] (L2) at (11,3) {};
\node[inner sep=1,fill=black,draw=black,circle] (L3) at (11,5) {};
\node[inner sep=1,fill=black,draw=black,circle] (L4) at (11,7) {};
\node[inner sep=1,fill=black,draw=black,circle] (L5) at (11,9) {};
\node[inner sep=1,fill=black,draw=black,circle,label=right:{$t_1$}] (M1) at (12,0) {};
\node[inner sep=1,fill=black,draw=black,circle,label=right:{$t_2$}] (M2) at (12,2) {};
\node[inner sep=1,fill=black,draw=black,circle,label=right:{$t_3$}] (M3) at (12,4) {};
\node[inner sep=1,fill=black,draw=black,circle,label=right:{$t_4$}] (M4) at (12,6) {};
\node[inner sep=1,fill=black,draw=black,circle,label=right:{$t_5$}] (M5) at (12,8) {};
\draw (A1) -- (B1) -- (C1) -- (D1) -- (E1) -- (F1) -- (G1) -- (H1) -- (I1) -- (J1) -- (K1) -- (L1) -- (M1);
\draw (A2) -- (B1) -- (C2) -- (D1) -- (E2) -- (F1) -- (G2) -- (H1) -- (I2) -- (J1) -- (K2) -- (L1) -- (M2);
\draw (A2) -- (B2) -- (C2) -- (D2) -- (E2) -- (F2) -- (G2) -- (H2) -- (I2) -- (J2) -- (K2) -- (L2) -- (M2);
\draw (A3) -- (B2) -- (C3) -- (D2) -- (E3) -- (F2) -- (G3) -- (H2) -- (I3) -- (J2) -- (K3) -- (L2) -- (M3);
\draw (A3) -- (B3) -- (C3) -- (D3) -- (E3) -- (F3) -- (G3) -- (H3) -- (I3) -- (J3) -- (K3) -- (L3) -- (M3);
\draw (A4) -- (B3) -- (C4) -- (D3) -- (E4) -- (F3) -- (G4) -- (H3) -- (I4) -- (J3) -- (K4) -- (L3) -- (M4);
\draw (A4) -- (B4) -- (C4) -- (D4) -- (E4) -- (F4) -- (G4) -- (H4) -- (I4) -- (J4) -- (K4) -- (L4) -- (M4);
\draw (A5) -- (B4) -- (C5) -- (D4) -- (E5) -- (F4) -- (G5) -- (H4) -- (I5) -- (J4) -- (K5) -- (L4) -- (M5);
\draw (A5) -- (B5) -- (C5) -- (D5) -- (E5) -- (F5) -- (G5) -- (H5) -- (I5) -- (J5) -- (K5) -- (L5) -- (M5);
\draw[dashed] (0,0) -- (0,9);
\draw[dashed] (12,0) -- (12,9);
\draw[very thick, red] (A5) -- (B5) -- (C5) -- (D4) -- (E4) -- (F4) -- (G5) -- (H5) -- (I5) -- (J4) -- (K4) -- (L4) -- (M5);
\draw[red,fill=red] (A5) circle (0.2);
\draw[red,fill=red] (B5) circle (0.2);
\draw[red,fill=red] (C5) circle (0.2);
\draw[red,fill=red] (D4) circle (0.2);
\draw[red,fill=red] (E4) circle (0.2);
\draw[red,fill=red] (F4) circle (0.2);
\draw[red,fill=red] (G5) circle (0.2);
\draw[red,fill=red] (H5) circle (0.2);
\draw[red,fill=red] (I5) circle (0.2);
\draw[red,fill=red] (J4) circle (0.2);
\draw[red,fill=red] (K4) circle (0.2);
\draw[red,fill=red] (L4) circle (0.2);
\draw[red,fill=red] (M5) circle (0.2);
\draw[very thick, blue] (A4) -- (B3) -- (C3) -- (D2) -- (E2) -- (F1) -- (G1) -- (H1) -- (I2) -- (J2) -- (K3) -- (L3) -- (M4);
\draw[blue,fill=blue] (A4) circle (0.2);
\draw[blue,fill=blue] (B3) circle (0.2);
\draw[blue,fill=blue] (C3) circle (0.2);
\draw[blue,fill=blue] (D2) circle (0.2);
\draw[blue,fill=blue] (E2) circle (0.2);
\draw[blue,fill=blue] (F1) circle (0.2);
\draw[blue,fill=blue] (G1) circle (0.2);
\draw[blue,fill=blue] (H1) circle (0.2);
\draw[blue,fill=blue] (I2) circle (0.2);
\draw[blue,fill=blue] (J2) circle (0.2);
\draw[blue,fill=blue] (K3) circle (0.2);
\draw[blue,fill=blue] (L3) circle (0.2);
\draw[blue,fill=blue] (M4) circle (0.2);
\end{tikzpicture} \vrule \begin{tikzpicture}[scale=0.42]
\node[inner sep=1,fill=black,draw=black,circle,label=left:{$s_1$}] (A1) at (0,0) {};
\node[inner sep=1,fill=black,draw=black,circle,label=left:{$s_2$}] (A2) at (0,2) {};
\node[inner sep=1,fill=black,draw=black,circle,label=left:{$s_3$}] (A3) at (0,4) {};
\node[inner sep=1,fill=black,draw=black,circle,label=left:{$s_4$}] (A4) at (0,6) {};
\node[inner sep=1,fill=black,draw=black,circle,label=left:{$s_5$}] (A5) at (0,8) {};
\node[inner sep=1,fill=black,draw=black,circle] (B1) at (1,1) {};
\node[inner sep=1,fill=black,draw=black,circle] (B2) at (1,3) {};
\node[inner sep=1,fill=black,draw=black,circle] (B3) at (1,5) {};
\node[inner sep=1,fill=black,draw=black,circle] (B4) at (1,7) {};
\node[inner sep=1,fill=black,draw=black,circle] (B5) at (1,9) {};
\node[inner sep=1,fill=black,draw=black,circle] (C1) at (2,0) {};
\node[inner sep=1,fill=black,draw=black,circle] (C2) at (2,2) {};
\node[inner sep=1,fill=black,draw=black,circle] (C3) at (2,4) {};
\node[inner sep=1,fill=black,draw=black,circle] (C4) at (2,6) {};
\node[inner sep=1,fill=black,draw=black,circle] (C5) at (2,8) {};
\node[inner sep=1,fill=black,draw=black,circle] (D1) at (3,1) {};
\node[inner sep=1,fill=black,draw=black,circle] (D2) at (3,3) {};
\node[inner sep=1,fill=black,draw=black,circle] (D3) at (3,5) {};
\node[inner sep=1,fill=black,draw=black,circle] (D4) at (3,7) {};
\node[inner sep=1,fill=black,draw=black,circle] (D5) at (3,9) {};
\node[inner sep=1,fill=black,draw=black,circle] (E1) at (4,0) {};
\node[inner sep=1,fill=black,draw=black,circle] (E2) at (4,2) {};
\node[inner sep=1,fill=black,draw=black,circle] (E3) at (4,4) {};
\node[inner sep=1,fill=black,draw=black,circle] (E4) at (4,6) {};
\node[inner sep=1,fill=black,draw=black,circle] (E5) at (4,8) {};
\node[inner sep=1,fill=black,draw=black,circle] (F1) at (5,1) {};
\node[inner sep=1,fill=black,draw=black,circle] (F2) at (5,3) {};
\node[inner sep=1,fill=black,draw=black,circle] (F3) at (5,5) {};
\node[inner sep=1,fill=black,draw=black,circle] (F4) at (5,7) {};
\node[inner sep=1,fill=black,draw=black,circle] (F5) at (5,9) {};
\node[inner sep=1,fill=black,draw=black,circle] (G1) at (6,0) {};
\node[inner sep=1,fill=black,draw=black,circle] (G2) at (6,2) {};
\node[inner sep=1,fill=black,draw=black,circle] (G3) at (6,4) {};
\node[inner sep=1,fill=black,draw=black,circle] (G4) at (6,6) {};
\node[inner sep=1,fill=black,draw=black,circle] (G5) at (6,8) {};
\node[inner sep=1,fill=black,draw=black,circle] (H1) at (7,1) {};
\node[inner sep=1,fill=black,draw=black,circle] (H2) at (7,3) {};
\node[inner sep=1,fill=black,draw=black,circle] (H3) at (7,5) {};
\node[inner sep=1,fill=black,draw=black,circle] (H4) at (7,7) {};
\node[inner sep=1,fill=black,draw=black,circle] (H5) at (7,9) {};
\node[inner sep=1,fill=black,draw=black,circle] (I1) at (8,0) {};
\node[inner sep=1,fill=black,draw=black,circle] (I2) at (8,2) {};
\node[inner sep=1,fill=black,draw=black,circle] (I3) at (8,4) {};
\node[inner sep=1,fill=black,draw=black,circle] (I4) at (8,6) {};
\node[inner sep=1,fill=black,draw=black,circle] (I5) at (8,8) {};
\node[inner sep=1,fill=black,draw=black,circle] (J1) at (9,1) {};
\node[inner sep=1,fill=black,draw=black,circle] (J2) at (9,3) {};
\node[inner sep=1,fill=black,draw=black,circle] (J3) at (9,5) {};
\node[inner sep=1,fill=black,draw=black,circle] (J4) at (9,7) {};
\node[inner sep=1,fill=black,draw=black,circle] (J5) at (9,9) {};
\node[inner sep=1,fill=black,draw=black,circle] (K1) at (10,0) {};
\node[inner sep=1,fill=black,draw=black,circle] (K2) at (10,2) {};
\node[inner sep=1,fill=black,draw=black,circle] (K3) at (10,4) {};
\node[inner sep=1,fill=black,draw=black,circle] (K4) at (10,6) {};
\node[inner sep=1,fill=black,draw=black,circle] (K5) at (10,8) {};
\node[inner sep=1,fill=black,draw=black,circle] (L1) at (11,1) {};
\node[inner sep=1,fill=black,draw=black,circle] (L2) at (11,3) {};
\node[inner sep=1,fill=black,draw=black,circle] (L3) at (11,5) {};
\node[inner sep=1,fill=black,draw=black,circle] (L4) at (11,7) {};
\node[inner sep=1,fill=black,draw=black,circle] (L5) at (11,9) {};
\node[inner sep=1,fill=black,draw=black,circle] (M1) at (12,0) {};
\node[inner sep=1,fill=black,draw=black,circle] (M2) at (12,2) {};
\node[inner sep=1,fill=black,draw=black,circle] (M3) at (12,4) {};
\node[inner sep=1,fill=black,draw=black,circle] (M4) at (12,6) {};
\node[inner sep=1,fill=black,draw=black,circle] (M5) at (12,8) {};
\node[inner sep=1,fill=black,draw=black,circle] (N1) at (13,1) {};
\node[inner sep=1,fill=black,draw=black,circle] (N2) at (13,3) {};
\node[inner sep=1,fill=black,draw=black,circle] (N3) at (13,5) {};
\node[inner sep=1,fill=black,draw=black,circle] (N4) at (13,7) {};
\node[inner sep=1,fill=black,draw=black,circle] (N5) at (13,9) {};
\node[inner sep=1,fill=black,draw=black,circle,label=right:{$t_1$}] (O1) at (14,0) {};
\node[inner sep=1,fill=black,draw=black,circle,label=right:{$t_2$}] (O2) at (14,2) {};
\node[inner sep=1,fill=black,draw=black,circle,label=right:{$t_3$}] (O3) at (14,4) {};
\node[inner sep=1,fill=black,draw=black,circle,label=right:{$t_4$}] (O4) at (14,6) {};
\node[inner sep=1,fill=black,draw=black,circle,label=right:{$t_5$}] (O5) at (14,8) {};
\draw (A1) -- (B1) -- (C1) -- (D1) -- (E1) -- (F1) -- (G1) -- (H1) -- (I1) -- (J1) -- (K1) -- (L1) -- (M1) -- (N1) -- (O1);
\draw (A2) -- (B1) -- (C2) -- (D1) -- (E2) -- (F1) -- (G2) -- (H1) -- (I2) -- (J1) -- (K2) -- (L1) -- (M2) -- (N1) -- (O2);
\draw (A2) -- (B2) -- (C2) -- (D2) -- (E2) -- (F2) -- (G2) -- (H2) -- (I2) -- (J2) -- (K2) -- (L2) -- (M2) -- (N2) -- (O2);
\draw (A3) -- (B2) -- (C3) -- (D2) -- (E3) -- (F2) -- (G3) -- (H2) -- (I3) -- (J2) -- (K3) -- (L2) -- (M3) -- (N2) -- (O3);
\draw (A3) -- (B3) -- (C3) -- (D3) -- (E3) -- (F3) -- (G3) -- (H3) -- (I3) -- (J3) -- (K3) -- (L3) -- (M3) -- (N3) -- (O3);
\draw (A4) -- (B3) -- (C4) -- (D3) -- (E4) -- (F3) -- (G4) -- (H3) -- (I4) -- (J3) -- (K4) -- (L3) -- (M4) -- (N3) -- (O4);
\draw (A4) -- (B4) -- (C4) -- (D4) -- (E4) -- (F4) -- (G4) -- (H4) -- (I4) -- (J4) -- (K4) -- (L4) -- (M4) -- (N4) -- (O4);
\draw (A5) -- (B4) -- (C5) -- (D4) -- (E5) -- (F4) -- (G5) -- (H4) -- (I5) -- (J4) -- (K5) -- (L4) -- (M5) -- (N4) -- (O5);
\draw (A5) -- (B5) -- (C5) -- (D5) -- (E5) -- (F5) -- (G5) -- (H5) -- (I5) -- (J5) -- (K5) -- (L5) -- (M5) -- (N5) -- (O5);
\draw[dashed] (1,0) -- (1,9);
\draw[dashed] (13,0) -- (13,9);
\draw[very thick, red] (A1) -- (B1) -- (C1) -- (D1) -- (E2) -- (F2) -- (G2) -- (H1) -- (I1) -- (J1) -- (K2) -- (L2) -- (M2) -- (N1) -- (O1);
\draw[red,fill=red] (A1) circle (0.2);
\draw[red,fill=red] (B1) circle (0.2);
\draw[red,fill=red] (C1) circle (0.2);
\draw[red,fill=red] (D1) circle (0.2);
\draw[red,fill=red] (E2) circle (0.2);
\draw[red,fill=red] (F2) circle (0.2);
\draw[red,fill=red] (G2) circle (0.2);
\draw[red,fill=red] (H1) circle (0.2);
\draw[red,fill=red] (I1) circle (0.2);
\draw[red,fill=red] (J1) circle (0.2);
\draw[red,fill=red] (K2) circle (0.2);
\draw[red,fill=red] (L2) circle (0.2);
\draw[red,fill=red] (M2) circle (0.2);
\draw[red,fill=red] (N1) circle (0.2);
\draw[red,fill=red] (O1) circle (0.2);
\draw[very thick, blue] (A2) -- (B2) -- (C3) -- (D3) -- (E4) -- (F4) -- (G5) -- (H5) -- (I5) -- (J4) -- (K4) -- (L3) -- (M3) -- (N2) -- (O2);
\draw[blue,fill=blue] (A2) circle (0.2);
\draw[blue,fill=blue] (B2) circle (0.2);
\draw[blue,fill=blue] (C3) circle (0.2);
\draw[blue,fill=blue] (D3) circle (0.2);
\draw[blue,fill=blue] (E4) circle (0.2);
\draw[blue,fill=blue] (F4) circle (0.2);
\draw[blue,fill=blue] (G5) circle (0.2);
\draw[blue,fill=blue] (H5) circle (0.2);
\draw[blue,fill=blue] (I5) circle (0.2);
\draw[blue,fill=blue] (J4) circle (0.2);
\draw[blue,fill=blue] (K4) circle (0.2);
\draw[blue,fill=blue] (L3) circle (0.2);
\draw[blue,fill=blue] (M3) circle (0.2);
\draw[blue,fill=blue] (N2) circle (0.2);
\draw[blue,fill=blue] (O2) circle (0.2);
\end{tikzpicture}
\caption{In the proof of \cref{thm:main}: the bijection from non-intersecting tuples of paths $\Pi\colon (s_4,s_5) \to (t_4,t_5)$ in $G^6$ (left) to non-intersecting tuples of paths $\Pi\colon (s_1,s_2) \to (t_1,t_2)$ in $G^7$ (right). The bijection flips the paths upside-down, and then adds to every path an up step at the beginning and a down step at the end.}
\label{fig:dyck_paths_flip}
\end{figure}

\begin{remark}
Jang et al.~\cite{jang2022negative} extended the investigations of Cigler and Krattenthaler~\cite{cigler2020bounded} by considering the ``negative versions'' of other types of bounded lattice paths, such as Motzkin and Schr\"{o}der paths. It would be interesting to check whether \cref{thm:main} continues to apply to these other lattice path families. We have not thought deeply about these extensions.
\end{remark}

\subsection{Reciprocity for Schur function evaluations with repeated values} \label{sec:schur}

In this section we follow the standard terminology and notation for symmetric functions as laid out for instance in \cite[Chapter~7]{stanley1999ec2}.

A \dfn{partition} $\lambda = (\lambda_1,\lambda_2,\ldots)$ is an infinite, weakly decreasing sequence of nonnegative integers $\lambda_1 \geq \lambda_2 \geq \cdots$ that is eventually zero. The \dfn{size} of $\lambda$ is $|\lambda| \coloneqq \lambda_1+\lambda_2 +\cdots$ and the \dfn{length} of $\lambda$ is $\ell(\lambda) \coloneqq \min \{i\geq 0\colon \lambda_{i+1}=0\}$. Often we leave off the zeroes and write simply $\lambda = (\lambda_1,\ldots,\lambda_{\ell(\lambda)})$. The \dfn{Young diagram} of $\lambda$ is the top- and left-justified two-dimensional array of boxes which has $\lambda_i$ boxes in the $i$th row. For instance, the Young diagram of $(3,3,1)$ is
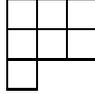
\[\ytableausetup{boxsize=1em} \ydiagram{3,3,1} \]
(Thus, we use ``English notation.'') The \dfn{transpose} (or \dfn{conjugate}) of $\lambda$, denoted $\lambda^t$, is the partition whose Young diagram is obtained from the Young diagram of $\lambda$ by reflecting across the main diagonal. For instance, the transpose of $(3,3,1)$ is $(3,2,2)$.

For $\lambda$ a partition, a \dfn{semistandard Young tableau (SSYT)} $T$ of shape $\lambda$ is a filling of the Young diagram of $\lambda$ with positive integers that is weakly increasing along rows and strictly increasing down columns. For instance, one SSYT of shape $(3,3,1)$ is
\[ T= \begin{ytableau}
1 & 1 & 2 \\
2 & 3 & 3 \\
5
\end{ytableau}\]
The \dfn{Schur function} $s_{\lambda}$ indexed by $\lambda$ is the generating function for SSYT of shape~$\lambda$:
\[ s_{\lambda}(x_1,x_2,\ldots) \coloneqq \sum_{\substack{\textrm{SSYT $T$}\\ \textrm{of shape $\lambda$}}} \, \prod_{i\geq 1} \, x_i^\textrm{$\#i$'s in $T$}.\] 
Schur functions are formal power series in infinitely many variables $x_1,x_2,\ldots$. 

For two partitions $\mu=(\mu_1,\mu_2,\ldots)$ and $\lambda=(\lambda_1,\lambda_2,\ldots)$ we write $\mu \subseteq \lambda$ to mean $\mu_i \leq \lambda_i$ for all $i\geq 1$. Equivalently, $\mu\subseteq \lambda$ if and only if the Young diagram of $\mu$ is contained in the Young diagram of $\lambda$. For partitions $\mu$, $\lambda$ with $\mu\subseteq \lambda$, the \dfn{skew shape} $\lambda/\mu$ is the set-theoretic difference of the Young diagrams of $\lambda$ and $\mu$. For instance, the skew shape $(3,3,2)/(1,1)$ is
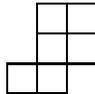
\[ \ydiagram{1+2,1+2,2}\]
We use $|\lambda/\mu| \coloneqq |\lambda| - |\mu|$ to denote the number of boxes in $\lambda/\mu$. A semistandard Young tableau of shape $\lambda/\mu$ is a filling of the skew shape $\lambda/\mu$ with positive integers that is weakly increasing along rows and strictly increasing down columns, and the \dfn{skew Schur function} $s_{\lambda/\mu}$ is the corresponding generating function:
\[ s_{\lambda/\mu}(x_1,x_2,\ldots) \coloneqq \sum_{\substack{\textrm{SSYT $T$}\\ \textrm{of shape $\lambda/\mu$}}} \, \prod_{i\geq 1} \, x_i^\textrm{$\#i$'s in $T$}.\] 

Let $\mathbf{z}=(z_1,\ldots,z_k)\in\mathbb{C}^k$ be a tuple of complex numbers. For any $n \geq 0$, we set~$\mathbf{z}^{n} \coloneqq (z_1,\ldots,z_k, \, z_1,\ldots,z_k,\, \ldots, \, z_1,\ldots,z_k) \in \mathbb{C}^{n\times k}$, i.e., each $z_i$ appears $n$ times in~$\mathbf{z}^n$. And we set $\mathbf{z}_{\mathrm{rev}} \coloneqq (z_k,\ldots,z_1)$, i.e.,  $\mathbf{z}_{\mathrm{rev}}$ is the reverse of $\mathbf{z}$. 

For such a tuple $\mathbf{z}=(z_1,\ldots,z_k) \in \mathbb{C}^k$, we use $s_{\lambda/\mu}(\mathbf{z}) \in \mathbb{C}$ to denote the result of specializing $x_i \coloneqq z_i$ for~$1\leq i \leq k$ and $x_i \coloneqq 0$ for $i > k$ in the Schur function~$s_{\lambda/\mu}$. Our goal in this section is to deduce the following reciprocity result for Schur function evaluations with repeated values.

\begin{thm} \label{thm:schur}
Let $\lambda/\mu$ be any skew shape and let $\mathbf{z}\in \mathbb{C}^k$. Then $s_{\lambda/\mu}(\mathbf{z}^n)$ satisfies a linear recurrence as a function of $n$, and for any $n \geq 1$ we have
\[s_{\lambda/\mu}(\mathbf{z}^{-n}) = (-1)^{|\lambda/\mu|} \, s_{\lambda^t/\mu^t}(\mathbf{z}_{\mathrm{rev}}^{n}).\]
\end{thm}

Of course, $ s_{\lambda^t/\mu^t}(\mathbf{z}_{\mathrm{rev}}^{n})= s_{\lambda^t/\mu^t}(\mathbf{z}^{n})$ since Schur functions are symmetric 
functions. But we wrote the statement of \cref{thm:schur} in the way we did because in our (first) proof of this theorem, via \cref{thm:main}, we will not use the fact that Schur functions are symmetric and it is $\mathbf{z}_{\mathrm{rev}}^{n}$ which arises naturally. Also, we will see later that $s_{\lambda/\mu}(\mathbf{z}^{n})$ is in fact a polynomial in $n$. But again, we said in \cref{thm:schur} merely that it satisfies a linear recurrence because we wish to apply \cref{thm:main}.

\begin{example}
For a partition $\lambda=(1^m)$ which is a single column, we have $s_{1^m} = e_m$, the \dfn{$m$th elementary symmetric function}
\[e_m \coloneqq \sum_{1 \leq i_1 < i_2 < \cdots < i_m} x_{i_1} x_{i_2} \cdots x_{i_m}.\]
Similarly, for a partition $\lambda = (m)$ which is a single row, we have $s_{m}=h_m$, the \dfn{$m$th complete homogeneous symmetric function}
\[h_m \coloneqq \sum_{1 \leq i_1 \leq i_2 \leq \cdots \leq i_m} x_{i_1} x_{i_2} \cdots x_{i_m}.\]
Hence, the single column/row version of \cref{thm:schur} says that
\[ e_m(\mathbf{z}^{-n}) = (-1)^m \, h_m(\mathbf{z}^n).\]
In particular, taking $\mathbf{z} = 1$ we have
\[e_m(1^n) = \binom{m}{n}, \qquad h_m(1^n) = \mch{m}{n}.\]
Thus, \cref{thm:schur} recovers the prototypical combinatorial reciprocity result~\eqref{eqn:ur_reciprocity}.
\end{example}

\begin{remark} \label{rem:order_poly}
Let us discuss the case $\mathbf{z}=1$ of \cref{thm:schur}. Then $s_{\lambda/\mu}(1^n)$ is the \dfn{order polynomial} $\Omega_{(P,\omega)}(n)$ for a certain labeled poset $(P,\omega)$ corresponding to the shape~$\lambda/\mu$: see~\cite[\S7.19]{stanley1999ec2} and~\cite[\S3.15]{stanley2012ec1}. In this case, \cref{thm:schur} follows from the order polynomial reciprocity theorem $\Omega_{(P,\omega)}(-n) = (-1)^{\#P} \, \Omega_{(P,\overline{\omega})} (n)$ \cite[Corollary~3.15.12]{stanley2012ec1}. For ordinary (i.e., non-skew) Schur functions $s_{\lambda}$, there is even an explicit product formula for the evaluation $s_{\lambda}(1^n)$ called the ``hook-content formula''~\cite[Corollary~7.21.4]{stanley1999ec2}. The reciprocity result $s_{\lambda}(1^{-n})=(-1)^{|\lambda|} \, s_{\lambda^t}(1^n)$ also follows immediately from the hook-content formula.
\end{remark}

\begin{remark} \label{rem:kronecker}
It is possible to use symmetric function theory to reduce \cref{thm:schur} to the case discussed in \cref{rem:order_poly}, as we now quickly explain. For simplicity, let us only consider ordinary Schur functions; since skew Schur functions are linear combinations of ordinary Schur functions, this is sufficient. Let $g_{\lambda,\mu,\nu}$ denote the \dfn{Kronecker coefficients}, which are the structure constants (in the basis of Schur functions) for the ``internal product'' of symmetric functions: $s_{\lambda} \ast s_{\mu} = \sum_{\nu} g_{\lambda,\mu,\nu} \, s_{\nu}$. Let $\mathbf{x}=x_1,x_2,\ldots$ and $\mathbf{y}=y_1,y_2,\ldots$ be two independent infinite sets of variables, and let $\mathbf{xy}=x_1y_1,x_1y_2,\ldots$ denote the infinite set of variables $x_iy_j$ for all $i,j \geq 1$. It is known (see, e.g., \cite[Exercise 7.78(c)]{stanley1999ec2}) that
\[ s_{\lambda}(\mathbf{xy}) = \sum_{\mu,\nu} g_{\lambda,\mu,\nu} \, s_{\mu}(\mathbf{x}) \, s_{\nu}(\mathbf{y}).\]
Let $\mathbf{z} \in \mathbb{C}^k$ and set $\mathbf{x}\coloneqq \mathbf{z}$ and $\mathbf{y} \coloneqq 1^n$. Then we obtain
\[ s_{\lambda}(\mathbf{z}^n) = \sum_{\mu,\nu} g_{\lambda,\mu,\nu} \, s_{\mu}(\mathbf{z}) \, s_{\nu}(1^n).\]
Hence, the $s_{\lambda}(\mathbf{z}^n)$ are linear combinations of the polynomials $s_{\nu}(1^n)$. In this way, we can indeed reduce \cref{thm:schur} to the case discussed in \cref{rem:order_poly}. However, as we will see below, there is actually a much more straightforward way to establish \cref{thm:schur} using symmetric function theory.
\end{remark}

We proceed to prove \cref{thm:schur} using \cref{thm:main}.

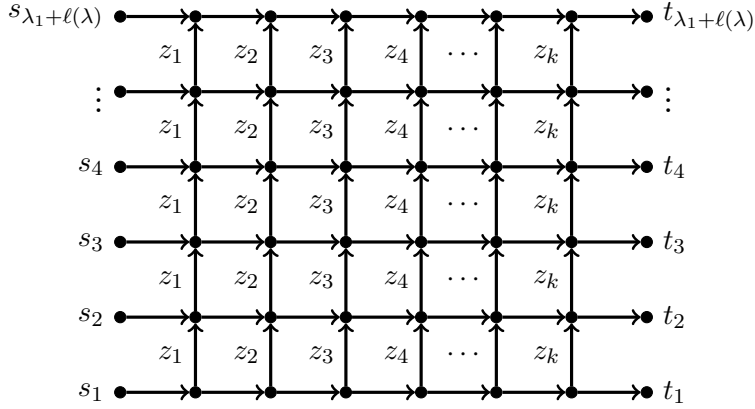
\begin{figure}
\begin{tikzpicture}
\node[inner sep=1.5,fill=black,draw=black,circle,label=left:{$s_1$}] (A1) at (0,0) {};
\node[inner sep=1.5,fill=black,draw=black,circle,label=left:{$s_2$}] (A2) at (0,1) {};
\node[inner sep=1.5,fill=black,draw=black,circle,label=left:{$s_3$}] (A3) at (0,2) {};
\node[inner sep=1.5,fill=black,draw=black,circle,label=left:{$s_4$}] (A4) at (0,3) {};
\node[inner sep=1.5,fill=black,draw=black,circle,label=left:{\Large $\vdots$}] (A5) at (0,4) {};
\node[inner sep=1.5,fill=black,draw=black,circle,label=left:{$s_{\lambda_1+\ell(\lambda)}$}] (A6) at (0,5) {};
\node[inner sep=1.5,fill=black,draw=black,circle] (B1) at (1,0) {};
\node[inner sep=1.5,fill=black,draw=black,circle] (B2) at (1,1) {};
\node[inner sep=1.5,fill=black,draw=black,circle] (B3) at (1,2) {};
\node[inner sep=1.5,fill=black,draw=black,circle] (B4) at (1,3) {};
\node[inner sep=1.5,fill=black,draw=black,circle] (B5) at (1,4) {};
\node[inner sep=1.5,fill=black,draw=black,circle] (B6) at (1,5) {};
\node[inner sep=1.5,fill=black,draw=black,circle] (C1) at (2,0) {};
\node[inner sep=1.5,fill=black,draw=black,circle] (C2) at (2,1) {};
\node[inner sep=1.5,fill=black,draw=black,circle] (C3) at (2,2) {};
\node[inner sep=1.5,fill=black,draw=black,circle] (C4) at (2,3) {};
\node[inner sep=1.5,fill=black,draw=black,circle] (C5) at (2,4) {};
\node[inner sep=1.5,fill=black,draw=black,circle] (C6) at (2,5) {};
\node[inner sep=1.5,fill=black,draw=black,circle] (D1) at (3,0) {};
\node[inner sep=1.5,fill=black,draw=black,circle] (D2) at (3,1) {};
\node[inner sep=1.5,fill=black,draw=black,circle] (D3) at (3,2) {};
\node[inner sep=1.5,fill=black,draw=black,circle] (D4) at (3,3) {};
\node[inner sep=1.5,fill=black,draw=black,circle] (D5) at (3,4) {};
\node[inner sep=1.5,fill=black,draw=black,circle] (D6) at (3,5) {};
\node[inner sep=1.5,fill=black,draw=black,circle] (E1) at (4,0) {};
\node[inner sep=1.5,fill=black,draw=black,circle] (E2) at (4,1) {};
\node[inner sep=1.5,fill=black,draw=black,circle] (E3) at (4,2) {};
\node[inner sep=1.5,fill=black,draw=black,circle] (E4) at (4,3) {};
\node[inner sep=1.5,fill=black,draw=black,circle] (E5) at (4,4) {};
\node[inner sep=1.5,fill=black,draw=black,circle] (E6) at (4,5) {};
\node[inner sep=1.5,fill=black,draw=black,circle] (F1) at (5,0) {};
\node[inner sep=1.5,fill=black,draw=black,circle] (F2) at (5,1) {};
\node[inner sep=1.5,fill=black,draw=black,circle] (F3) at (5,2) {};
\node[inner sep=1.5,fill=black,draw=black,circle] (F4) at (5,3) {};
\node[inner sep=1.5,fill=black,draw=black,circle] (F5) at (5,4) {};
\node[inner sep=1.5,fill=black,draw=black,circle] (F6) at (5,5) {};
\node[inner sep=1.5,fill=black,draw=black,circle] (G1) at (6,0) {};
\node[inner sep=1.5,fill=black,draw=black,circle] (G2) at (6,1) {};
\node[inner sep=1.5,fill=black,draw=black,circle] (G3) at (6,2) {};
\node[inner sep=1.5,fill=black,draw=black,circle] (G4) at (6,3) {};
\node[inner sep=1.5,fill=black,draw=black,circle] (G5) at (6,4) {};
\node[inner sep=1.5,fill=black,draw=black,circle] (G6) at (6,5) {};
\node[inner sep=1.5,fill=black,draw=black,circle,label=right:{$t_1$}] (H1) at (7,0) {};
\node[inner sep=1.5,fill=black,draw=black,circle,label=right:{$t_2$}] (H2) at (7,1) {};
\node[inner sep=1.5,fill=black,draw=black,circle,label=right:{$t_3$}] (H3) at (7,2) {};
\node[inner sep=1.5,fill=black,draw=black,circle,label=right:{$t_4$}] (H4) at (7,3) {};
\node[inner sep=1.5,fill=black,draw=black,circle,label=right:{\Large $\vdots$}] (H5) at (7,4) {};
\node[inner sep=1.5,fill=black,draw=black,circle,label=right:{$t_{\lambda_1+\ell(\lambda)}$}] (H6) at (7,5) {};
\draw[->,very thick] (A1) edge (B1) (B1) edge (C1) (C1) edge (D1) (D1) edge (E1) (E1) edge (F1) (F1) edge (G1) (G1) edge (H1);
\draw[->,very thick] (A2) edge (B2) (B2) edge (C2) (C2) edge (D2) (D2) edge (E2) (E2) edge (F2) (F2) edge (G2) (G2) edge (H2);
\draw[->,very thick] (A3) edge (B3) (B3) edge (C3) (C3) edge (D3) (D3) edge (E3) (E3) edge (F3) (F3) edge (G3) (G3) edge (H3);
\draw[->,very thick] (A4) edge (B4) (B4) edge (C4) (C4) edge (D4) (D4) edge (E4) (E4) edge (F4) (F4) edge (G4) (G4) edge (H4);
\draw[->,very thick] (A5) edge (B5) (B5) edge (C5) (C5) edge (D5) (D5) edge (E5) (E5) edge (F5) (F5) edge (G5) (G5) edge (H5);
\draw[->,very thick] (A6) edge (B6) (B6) edge (C6) (C6) edge (D6) (D6) edge (E6) (E6) edge (F6) (F6) edge (G6) (G6) edge (H6);
\draw[->,very thick] (B1) edge node[left] {$z_1$} (B2) (B2) edge node[left] {$z_1$} (B3) (B3) edge node[left] {$z_1$} (B4) (B4) edge node[left] {$z_1$} (B5) (B5) edge node[left] {$z_1$} (B6);
\draw[->,very thick] (C1) edge node[left] {$z_2$} (C2) (C2) edge node[left] {$z_2$} (C3) (C3) edge node[left] {$z_2$} (C4) (C4) edge node[left] {$z_2$} (C5) (C5) edge node[left] {$z_2$} (C6);
\draw[->,very thick] (D1) edge node[left] {$z_3$} (D2) (D2) edge node[left] {$z_3$} (D3) (D3) edge node[left] {$z_3$} (D4) (D4) edge node[left] {$z_3$} (D5) (D5) edge node[left] {$z_3$} (D6);
\draw[->,very thick] (E1) edge node[left] {$z_4$} (E2) (E2) edge node[left] {$z_4$} (E3) (E3) edge node[left] {$z_4$} (E4) (E4) edge node[left] {$z_4$} (E5) (E5) edge node[left] {$z_4$} (E6);
\draw[->,very thick] (F1) edge node[left] {$\cdots$} (F2) (F2) edge node[left] {$\cdots$} (F3) (F3) edge node[left] {$\cdots$} (F4) (F4) edge node[left] {$\cdots$} (F5) (F5) edge node[left] {$\cdots$} (F6);
\draw[->,very thick] (G1) edge node[left] {$z_k$} (G2) (G2) edge node[left] {$z_k$} (G3) (G3) edge node[left] {$z_k$} (G4) (G4) edge node[left] {$z_k$} (G5) (G5) edge node[left] {$z_k$} (G6);
\end{tikzpicture}
\caption{The acyclic planar network used in the proof of \cref{thm:schur}.}
\label{fig:schur}
\end{figure}

\begin{proof}[Proof of \cref{thm:schur}]
Let $G$ be the acyclic planar network depicted in \cref{fig:schur}. Notice that $G$ has sources $s_1,\ldots,s_{\lambda_1+\ell(\lambda)}$ and sinks $t_1,\ldots,t_{\lambda_1+\ell(\lambda)}$, and the weight of each horizontal edges is one while the weight of a vertical edge in the $i$th column is $z_i$. There is a unique non-intersecting tuple of paths in $G$ connecting all of the sources to all of the sinks, and all edges in this tuple of paths have weight one. Hence, as discussed in \cref{rem:determinant}, we have that $\det(\mathsf{P}_G)=1$. 

It is well known that non-intersecting tuples of paths in $G$ are in bijection with semistandard Young tableaux: see, e.g.,~\cite[1st proof of Theorem 7.16.1]{stanley1999ec2}. Specifically, with $\ell\coloneqq \ell(\lambda)$, let $I=\{i_1 < \cdots < i_\ell\} \coloneqq \{\mu_{\ell}+1,\mu_{\ell-1}+2,\ldots,\mu_1+\ell\}$ and $J=\{j_1 < \cdots < j_{\ell}\} \coloneqq \{\lambda_{\ell}+1,\lambda_{\ell-1}+2,\ldots,\lambda_1+\ell\}$. Then non-intersecting tuples of paths $\Pi\colon(s_{i_1},\ldots,s_{i_\ell})\to(t_{j_1},\ldots,t_{j_\ell})$ in $G$ correspond to SSYTs of shape~$\lambda/\mu$ with entries in $[k]$: the bijection is depicted in \cref{fig:schur_bij}. The bijection is weight-preserving and hence, summing over such~$\Pi$, we obtain $\sum_{\Pi} w(\Pi)=s_{\lambda/\mu}(\mathbf{z})$. What is more, we have $f_G(I,J;n) = s_{\lambda/\mu}(\mathbf{z}^n)$ for exactly the same reason. So, by \cref{thm:main}, $s_{\lambda/\mu}(\mathbf{z}^n)$ satisfies a linear recurrence as a function of $n$.

Moreover, by \cref{thm:main}, $s_{\lambda/\mu}(\mathbf{z}^{-n})= (-1)^{|\lambda/\mu|} \, f_{G}([\lambda_1+\ell]\setminus J,[\lambda_1+\ell]\setminus I;n)$. Now, we would like to say that non-intersecting tuples of paths in $G$ connecting sources indexed by $[\lambda_1+\ell]\setminus J$ to sinks indexed by $[\lambda_1+\ell]\setminus I$ correspond to SSYTs of shape $\lambda^{t}/\mu^{t}$. But this is not quite right: they in fact correspond to SSYTs whose shape is the $180^\circ$ rotation of $\lambda^{t}/\mu^{t}$. Nevertheless, given a tableau $T$ with entries in~$[k]$ we can obtain another tableau $\widehat{T}$ by rotating $T$ $180^{\circ}$ degrees and replacing each entry~$i$ with $k+1-i$. Since this involution $T\mapsto \widehat{T}$ reverses the weight of the tableau, we obtain $f_{G}([\lambda_1+\ell]\setminus J,[\lambda_1+\ell]\setminus I;n) = s_{\lambda^t/\mu^t}(\mathbf{z}_{\mathrm{rev}}^n)$. Therefore, we conclude $s_{\lambda/\mu}(\mathbf{z}^{-n})= (-1)^{|\lambda/\mu|} \, s_{\lambda^t/\mu^t}(\mathbf{z}_{\mathrm{rev}}^n)$, as desired.
\end{proof}

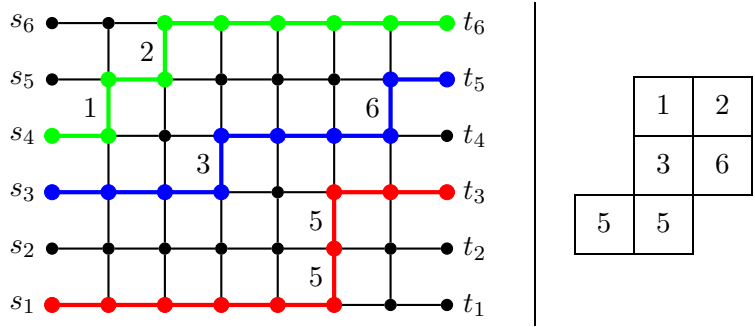
\begin{figure}
\begin{tikzpicture}[scale=0.75]
\node[inner sep=1.5,fill=black,draw=black,circle,label=left:{$s_1$}] (A1) at (0,0) {};
\node[inner sep=1.5,fill=black,draw=black,circle,label=left:{$s_2$}] (A2) at (0,1) {};
\node[inner sep=1.5,fill=black,draw=black,circle,label=left:{$s_3$}] (A3) at (0,2) {};
\node[inner sep=1.5,fill=black,draw=black,circle,label=left:{$s_4$}] (A4) at (0,3) {};
\node[inner sep=1.5,fill=black,draw=black,circle,label=left:{$s_5$}] (A5) at (0,4) {};
\node[inner sep=1.5,fill=black,draw=black,circle,label=left:{$s_6$}] (A6) at (0,5) {};
\node[inner sep=1.5,fill=black,draw=black,circle] (B1) at (1,0) {};
\node[inner sep=1.5,fill=black,draw=black,circle] (B2) at (1,1) {};
\node[inner sep=1.5,fill=black,draw=black,circle] (B3) at (1,2) {};
\node[inner sep=1.5,fill=black,draw=black,circle] (B4) at (1,3) {};
\node[inner sep=1.5,fill=black,draw=black,circle] (B5) at (1,4) {};
\node[inner sep=1.5,fill=black,draw=black,circle] (B6) at (1,5) {};
\node[inner sep=1.5,fill=black,draw=black,circle] (C1) at (2,0) {};
\node[inner sep=1.5,fill=black,draw=black,circle] (C2) at (2,1) {};
\node[inner sep=1.5,fill=black,draw=black,circle] (C3) at (2,2) {};
\node[inner sep=1.5,fill=black,draw=black,circle] (C4) at (2,3) {};
\node[inner sep=1.5,fill=black,draw=black,circle] (C5) at (2,4) {};
\node[inner sep=1.5,fill=black,draw=black,circle] (C6) at (2,5) {};
\node[inner sep=1.5,fill=black,draw=black,circle] (D1) at (3,0) {};
\node[inner sep=1.5,fill=black,draw=black,circle] (D2) at (3,1) {};
\node[inner sep=1.5,fill=black,draw=black,circle] (D3) at (3,2) {};
\node[inner sep=1.5,fill=black,draw=black,circle] (D4) at (3,3) {};
\node[inner sep=1.5,fill=black,draw=black,circle] (D5) at (3,4) {};
\node[inner sep=1.5,fill=black,draw=black,circle] (D6) at (3,5) {};
\node[inner sep=1.5,fill=black,draw=black,circle] (E1) at (4,0) {};
\node[inner sep=1.5,fill=black,draw=black,circle] (E2) at (4,1) {};
\node[inner sep=1.5,fill=black,draw=black,circle] (E3) at (4,2) {};
\node[inner sep=1.5,fill=black,draw=black,circle] (E4) at (4,3) {};
\node[inner sep=1.5,fill=black,draw=black,circle] (E5) at (4,4) {};
\node[inner sep=1.5,fill=black,draw=black,circle] (E6) at (4,5) {};
\node[inner sep=1.5,fill=black,draw=black,circle] (F1) at (5,0) {};
\node[inner sep=1.5,fill=black,draw=black,circle] (F2) at (5,1) {};
\node[inner sep=1.5,fill=black,draw=black,circle] (F3) at (5,2) {};
\node[inner sep=1.5,fill=black,draw=black,circle] (F4) at (5,3) {};
\node[inner sep=1.5,fill=black,draw=black,circle] (F5) at (5,4) {};
\node[inner sep=1.5,fill=black,draw=black,circle] (F6) at (5,5) {};
\node[inner sep=1.5,fill=black,draw=black,circle] (G1) at (6,0) {};
\node[inner sep=1.5,fill=black,draw=black,circle] (G2) at (6,1) {};
\node[inner sep=1.5,fill=black,draw=black,circle] (G3) at (6,2) {};
\node[inner sep=1.5,fill=black,draw=black,circle] (G4) at (6,3) {};
\node[inner sep=1.5,fill=black,draw=black,circle] (G5) at (6,4) {};
\node[inner sep=1.5,fill=black,draw=black,circle] (G6) at (6,5) {};
\node[inner sep=1.5,fill=black,draw=black,circle,label=right:{$t_1$}] (H1) at (7,0) {};
\node[inner sep=1.5,fill=black,draw=black,circle,label=right:{$t_2$}] (H2) at (7,1) {};
\node[inner sep=1.5,fill=black,draw=black,circle,label=right:{$t_3$}] (H3) at (7,2) {};
\node[inner sep=1.5,fill=black,draw=black,circle,label=right:{$t_4$}] (H4) at (7,3) {};
\node[inner sep=1.5,fill=black,draw=black,circle,label=right:{$t_5$}] (H5) at (7,4) {};
\node[inner sep=1.5,fill=black,draw=black,circle,label=right:{$t_6$}] (H6) at (7,5) {};
\draw[thick] (A1) edge (B1) (B1) edge (C1) (C1) edge (D1) (D1) edge (E1) (E1) edge (F1) (F1) edge (G1) (G1) edge (H1);
\draw[thick] (A2) edge (B2) (B2) edge (C2) (C2) edge (D2) (D2) edge (E2) (E2) edge (F2) (F2) edge (G2) (G2) edge (H2);
\draw[thick] (A3) edge (B3) (B3) edge (C3) (C3) edge (D3) (D3) edge (E3) (E3) edge (F3) (F3) edge (G3) (G3) edge (H3);
\draw[thick] (A4) edge (B4) (B4) edge (C4) (C4) edge (D4) (D4) edge (E4) (E4) edge (F4) (F4) edge (G4) (G4) edge (H4);
\draw[thick] (A5) edge (B5) (B5) edge (C5) (C5) edge (D5) (D5) edge (E5) (E5) edge (F5) (F5) edge (G5) (G5) edge (H5);
\draw[thick] (A6) edge (B6) (B6) edge (C6) (C6) edge (D6) (D6) edge (E6) (E6) edge (F6) (F6) edge (G6) (G6) edge (H6);
\draw[thick] (B1) edge (B2) (B2) edge (B3) (B3) edge (B4) (B4) edge node[left] {$1$} (B5) (B5) edge (B6);
\draw[thick] (C1) edge (C2) (C2) edge (C3) (C3) edge (C4) (C4) edge (C5) (C5) edge node[left] {$2$} (C6);
\draw[thick] (D1) edge (D2) (D2) edge (D3) (D3) edge node[left] {$3$} (D4) (D4) edge (D5) (D5) edge (D6);
\draw[thick] (E1) edge (E2) (E2) edge (E3) (E3) edge (E4) (E4) edge (E5) (E5) edge (E6);
\draw[thick] (F1) edge node[left] {$5$} (F2) (F2) edge node[left] {$5$} (F3) (F3) edge (F4) (F4) edge (F5) (F5) edge (F6);
\draw[thick] (G1) edge (G2) (G2) edge (G3) (G3) edge (G4) (G4) edge node[left] {$6$} (G5) (G5) edge (G6);
\draw[ultra thick,red] (A1) -- (B1) -- (C1) -- (D1) --(E1) -- (F1) -- (F2) -- (F3) -- (G3) -- (H3);
\node[inner sep=2,fill=red,draw=red,circle] at (A1) {};
\node[inner sep=2,fill=red,draw=red,circle] at (B1) {};
\node[inner sep=2,fill=red,draw=red,circle] at (C1) {};
\node[inner sep=2,fill=red,draw=red,circle] at (D1) {};
\node[inner sep=2,fill=red,draw=red,circle] at (E1) {};
\node[inner sep=2,fill=red,draw=red,circle] at (F1) {};
\node[inner sep=2,fill=red,draw=red,circle] at (F2) {};
\node[inner sep=2,fill=red,draw=red,circle] at (F3) {};
\node[inner sep=2,fill=red,draw=red,circle] at (G3) {};
\node[inner sep=2,fill=red,draw=red,circle] at (H3) {};
\draw[ultra thick,blue] (A3) -- (B3) -- (C3) -- (D3) -- (D4) -- (E4) -- (F4) -- (G4) -- (G5) -- (H5);
\node[inner sep=2,fill=blue,draw=blue,circle] at (A3) {};
\node[inner sep=2,fill=blue,draw=blue,circle] at (B3) {};
\node[inner sep=2,fill=blue,draw=blue,circle] at (C3) {};
\node[inner sep=2,fill=blue,draw=blue,circle] at (D3) {};
\node[inner sep=2,fill=blue,draw=blue,circle] at (D4) {};
\node[inner sep=2,fill=blue,draw=blue,circle] at (E4) {};
\node[inner sep=2,fill=blue,draw=blue,circle] at (F4) {};
\node[inner sep=2,fill=blue,draw=blue,circle] at (G4) {};
\node[inner sep=2,fill=blue,draw=blue,circle] at (G5) {};
\node[inner sep=2,fill=blue,draw=blue,circle] at (H5) {};
\draw[ultra thick,green] (A4) -- (B4) -- (B5) -- (C5) -- (C6) -- (D6) -- (E6) -- (F6) -- (G6) -- (H6);
\node[inner sep=2,fill=green,draw=green,circle] at (A4) {};
\node[inner sep=2,fill=green,draw=green,circle] at (B4) {};
\node[inner sep=2,fill=green,draw=green,circle] at (B5) {};
\node[inner sep=2,fill=green,draw=green,circle] at (C5) {};
\node[inner sep=2,fill=green,draw=green,circle] at (C6) {};
\node[inner sep=2,fill=green,draw=green,circle] at (D6) {};
\node[inner sep=2,fill=green,draw=green,circle] at (E6) {};
\node[inner sep=2,fill=green,draw=green,circle] at (F6) {};
\node[inner sep=2,fill=green,draw=green,circle] at (G6) {};
\node[inner sep=2,fill=green,draw=green,circle] at (H6) {};
\end{tikzpicture} \quad \vrule \quad \ytableausetup{boxsize=2em}\begin{tikzpicture} \node at (0,0) {$\begin{ytableau}
\none & 1 & 2 \\
\none & 3 & 6 \\
5 & 5
\end{ytableau}$};
\node at (0,-2) {};
\end{tikzpicture}
\caption{In the proof of \cref{thm:schur}: the bijection from non-intersecting tuples of paths $\Pi\colon (s_1,s_3,s_4)\to(t_3,t_5,t_6)$ in $G$ (left) to SSYTs of shape $(3,2,2)/(1,1)$ (right). Under this bijection, each row of the SSYT records the vertical steps of the corresponding path.}
\label{fig:schur_bij}
\end{figure}

Actually, \cref{thm:schur} is a corollary of a more general reciprocity result for arbitrary symmetric functions, as we now explain.

Let $\Lambda$ denote the ring of symmetric functions. The elements $f\in \Lambda$ are bounded degree formal power series $f=\sum_{\alpha=(\alpha_1,\alpha_2,\ldots)\in\mathbb{N}^{\infty}}c_{\alpha} \, x_1^{\alpha_1}x_2^{\alpha_2}\cdots$ in infinitely many variables $x_1,x_2,\ldots$ that are invariant under arbitrary permutation of the variables. Following~\cite[Chapter~7]{stanley1999ec2}, we take the ring of coefficients of $\Lambda$ to be the rational numbers, i.e., $c_\alpha\in\mathbb{Q}$ for all $\alpha$. Hence, $\Lambda$ is a $\mathbb{Q}$-algebra. 

The ring of symmetric functions $\Lambda$ is a polynomial ring in the elementary symmetric functions $e_m$, and it is also a polynomial ring in the complete homogeneous symmetric functions $h_m$. Moreover, there is a canonical involutive algebra automorphism $\omega\colon \Lambda \to \Lambda$ defined by $\omega(e_m) \coloneqq h_m$ for all $m \geq 0$.

For any symmetric function $f \in \Lambda$ and $\mathbf{z}\in\mathbb{C}^k$, we let~$f(\mathbf{z}) \in \mathbb{C}$ denote the result of specializing $x_i \coloneqq z_i$ for $1 \leq i \leq k$ and $x_i \coloneqq 0$ for $i > k$ in~$f$. We have the following known reciprocity result concerning evaluations with repeated values for any symmetric function.

\begin{thm}[{\cite[Exercise 23]{stanley2022ec2supp}}] \label{thm:sym}
Let $f \in \Lambda$ be any symmetric function and let $\mathbf{z} \in \mathbb{C}^k$. Then $f(\mathbf{z}^n)$ is a polynomial in $n$. Moreover, suppose that $f$ is homogeneous of degree $m$. Then $f(\mathbf{z}^{-n}) = (-1)^m\, (\omega f)(\mathbf{z}^n)$.
\end{thm} 

\Cref{thm:schur} is an immediate corollary of \cref{thm:sym} because Schur functions are symmetric functions and $\omega s_{\lambda/\mu} = s_{\lambda^t/\mu^t}$ (see \cite[Theorem~7.15.6]{stanley1999ec2}). In fact, the~$s_{\lambda}$ form a basis of $\Lambda$, so \cref{thm:schur} is more-or-less equivalent to \cref{thm:sym}.

As we mentioned, \cref{thm:sym} is a known result: it is one of Richard Stanley's supplementary exercises for Chapter~7 of \emph{Enumerative Combinatorics, Vol.~2}~\cite{stanley2022ec2supp}. Since the proof of this theorem is very short, we include it here.

\begin{proof}[Proof of \cref{thm:sym}]
For $m \geq 0$, let $p_m \in \Lambda$ denote the \dfn{$m$th power sum symmetric function}:
\[p_m(x_1,x_2,\ldots)= \sum_{i \geq 1} x_i^m. \]
And for a partition $\lambda=(\lambda_1,\lambda_2,\ldots)$, let $p_{\lambda}\coloneqq p_{\lambda_1} p_{\lambda_2} \cdots$. It is well known that the $p_{\lambda}$ form a basis of $\Lambda$ compatible with its decomposition into homogeneous components (see \cite[Corollary~7.7.2]{stanley1999ec2}). Therefore, it suffices to verify \cref{thm:sym} for the~$p_{\lambda}$.

It is clear that $p_m(\mathbf{z}^n) = n\cdot p_m(\mathbf{z})$ for any $m\geq 1$. Thus, $p_{\lambda}(\mathbf{z}^n) = n^{\ell(\lambda)} \cdot p_{\lambda}(\mathbf{z})$. So indeed $p_{\lambda}(\mathbf{z}^n)$ is a polynomial in $n$. 
 Moreover, $p_{\lambda}(\mathbf{z}^{-n})=(-1)^{\ell(\lambda)} \, p_{\lambda}(\mathbf{z}^n)$. And it is well known that $\omega p_{\lambda} = (-1)^{|\lambda|-\ell(\lambda)} \, p_{\lambda}$ (see~\cite[Proposition~7.7.5]{stanley1999ec2}). So indeed we have $p_{\lambda}(\mathbf{z}^{-n})=(-1)^{|\lambda|} \, (\omega p_{\lambda})(\mathbf{z}^{n})$. This completes the verification of \cref{thm:sym} for the~$p_{\lambda}$, and so the theorem is proved for all symmetric functions.
\end{proof}

In spite of the simplicity of this proof of \cref{thm:sym}, we also think the proof we gave above of \cref{thm:schur} using \cref{thm:main} was worthwhile because of its more combinatorial nature.

\begin{remark}
In this remark we discuss \dfn{$(P,\omega)$-partitions}, which we do not define. Consult \cite[\S3.15]{stanley2012ec1} for the relevant definitions. Let $(P,\omega)$ be any finite labeled poset. The order polynomial $\Omega_{(P,\omega)}(n)$ counts the number of $(P,\omega)$-partitions $\sigma\colon P \to [n]$ and, as mentioned in \cref{rem:order_poly}, we have the order polynomial reciprocity theorem $\Omega_{(P,\omega)}(-n) = (-1)^{\#P} \, \Omega_{(P,\overline{\omega})}(n)$. But we can also form the generating function
\[K_{(P,\omega)}(x_1,x_2,\ldots) \coloneqq \sum_{\sigma} \prod_{i \geq 1} x_i^{\#\{p\in P\colon \sigma(p)=i\}},\]
a sum over all $(P,\omega)$-partitions $\sigma\colon P \to \{1,2,\ldots\}$. For the labeled poset $(P,\omega)$ corresponding to a skew shape $\lambda/\mu$ we have $K_{(P,\omega)}=s_{\lambda/\mu}$. However, in general, $K_{(P,\omega)}$ is only a \dfn{quasisymmetric function} (see~\cite[\S7.19]{stanley1999ec2}). Nevertheless, it still makes sense to consider the specialization $K_{(P,\omega)}(\mathbf{z}) \in \mathbb{C}$ for any $\mathbf{z}\in\mathbb{C}^k$. We could hope for a reciprocity theorem for the repeated values evaluation $K_{(P,\omega)}(\mathbf{z}^n)$ which would recover the order polynomial reciprocity theorem in the case of $\mathbf{z}=1$. Such a reciprocity theorem does indeed exist, and in even more generality, as we explain in the next paragraph.

There is a direct generalization of \cref{thm:sym} to all quasisymmetric functions. The algebra of quasisymmetric functions also has a canonical involution $\omega$. Thus, the statement of \cref{thm:sym} makes sense if we replace ``symmetric function~$f\in \Lambda$'' with ``quasisymmetric function $f$,'' and with this change the statement remains true. In fact, for any graded \dfn{Hopf algebra} $A$ equipped with a character map $\zeta\colon A \to \mathbf{k}$ there is a reciprocity result for the ``repeated character evaluation'' $\zeta^{\otimes m} \circ \Delta^{m-1}$: see \cite[Proposition~7.1.7]{grinberg2014hopf}. Taking $\zeta$ to be the character on the Hopf algebra $\mathrm{QSym}$ of quasisymmetric functions which sends $f$ to $f(\mathbf{z})\in\mathbb{C}$ yields the desired extension of \cref{thm:sym}.

Finally, we note that another, related approach to proving the quasisymmetric generalization of \cref{thm:sym} is to extend the argument sketched in \cref{rem:kronecker}, using the ``internal coproduct'' of quasisymmetric functions as defined by Gessel in~\cite{gessel1984multipartite}.
\end{remark}

\bibliography{main}{}
\bibliographystyle{abbrv}

\end{document}